\documentclass{amsart}
\usepackage{hyperref}
\usepackage{graphicx,color}
\usepackage[section]{algorithm}
\usepackage{algorithmic}
\usepackage{enumerate}
\usepackage{mathrsfs,enumerate}
\usepackage{multirow,array}
\usepackage{subfigure}
\usepackage{pstricks,pst-node,pst-tree, pst-plot, pst-eps}
\usepackage{xspace}
\usepackage{marginnote}

\def\cS{{\mathcal S}}

\def\cF{{\mathcal F}}

\def\dH#1{\dot{H}^{#1}}
\def\cT{{\mathcal T}}
\def\cQ{{\mathcal Q}}

\def\Forall{\qquad \hbox{for all }}

\def\beq#1\eeq{\begin{equation} #1 \end{equation}}
\def\bal#1\eal{\begin{aligned} #1 \end{aligned}}

\def\RR{{\mathbb R}}
\def\CC{{\mathbb C}}

\def\tH#1{\widetilde H^{#1}}

\chardef\atsign='100

\def\tH{\widetilde H}
\def\tHs{{\tH^s}}
\def\tHr{{\tH^r}}

\def\RRD{{\RR^d}}
\def\OMt{{\Omega^M(t)}}

\def\pow{{s}}

\def\Nplus{{N^+}}
\def\Nminus{{N^-}}
\def\vh{\mathbb V_h}
\def\omt{{\Omega^M(t)}}
\def\ltomt{{L^2(\omt)}}
\def\tzer{S}

\def\dd{\mathsf{d}}

\def\Thd{{\mathcal{T}_h(D)}}
\def\Tht{\widetilde{\mathcal{T}}_h(D)}
\def\ttau{{\widetilde \tau}}
\def\Thdd{\widetilde{\mathcal{T}}_h(D^\delta_h)}

\renewcommand{\Re}{\mathfrak{Re}}
\renewcommand{\Im}{\mathfrak{Im}}

\theoremstyle{plain}
\newtheorem{theorem}{Theorem}[section]
\newtheorem{lemma}[theorem]{Lemma}

\theoremstyle{theorem}

\newtheorem{remark}{Remark}[section]

\graphicspath{{figures/}}

\begin{document}

\title[APPROXIMATION OF THE FRACTIONAL LAPLACIAN]
{Numerical Approximation of the Integral Fractional Laplacian}
\author{Andrea Bonito}
\address{Department of Mathematics, Texas A\&M University, College Station,
TX~77843-3368.}
\email{bonito\atsign math.tamu.edu}

\author{Wenyu Lei}
\address{Mathematics Area, SISSA - Scuola Internazionale Superiore di Studi Avanzati, via Bonomea, 265, 34136 Trieste Italy}
\email{wenyu.lei\atsign sissa.it}

\author{Joseph E. Pasciak}
\address{Department of Mathematics, Texas A\&M University, College Station,
TX~77843-3368.}
\email{pasciak\atsign math.tamu.edu}

\date{\today}

\begin{abstract} 
We propose a new nonconforming finite element algorithm to approximate the solution to
the elliptic problem involving the fractional Laplacian.
We first derive an integral representation of the bilinear form corresponding to the variational problem.
The numerical approximation of the action of the corresponding stiffness
matrix consists of  three steps: (i) apply a sinc quadrature scheme to
approximate the integral representation by a finite sum where each term
involves the solution of an elliptic partial differential equation
defined on the entire space, (ii) truncate each elliptic problem to a
bounded domain,  (iii) use the finite element method for the space
approximation 
on each  truncated domain. The consistency error analysis for the three 
steps is discussed together with the numerical implementation of the
entire algorithm. The results of computations are given 
illustrating the error behavior in terms of the mesh size of the physical domain, 
the domain truncation parameter and the quadrature spacing parameter.
\end{abstract} 


\maketitle

\section{Introduction}

We consider a nonlocal model  on a bounded domain involving the Riesz
fractional derivative (i.e., the fractional Laplacian).
For theory and numerical analysis of  
general nonlocal models, we refer to the review paper \cite{DGLZ12} and references therein. 
Particularly, several applications are modeled by partial differential equations involving the fractional 
Laplacian: obstacle problems from symmetric $\alpha$-stable L\'evy processes \cite{Tankov03,Levendorskii04,Pham97};
image denoisings \cite{GH15}; fractional kinetics and anomalous transport \cite{Zaslavsky02};
fractal conservation laws \cite{Droniou10,BKW01}; and geophysical fluid dynamics \cite{CC03,Constantin02,CCW01,HPGS95}.

In this paper, we consider a class of fractional boundary problems on
bounded domains where
the fractional derivative comes from the fractional Laplacian defined on
all  of $\mathbb R^d$.  The motivation for these problems is illustrated by 
an evolution equation considered by Meuller \cite{mueller} of the form:
\begin{align}
  u_t &= -\widetilde\Lambda_su  +f(u),\quad\text{in } \RR^+\times
  D,\label{eq1}\\ 
u&=0,\quad \hbox{ in } D^c.  \label{eq2}
\end{align}
Here  $D$ is a convex polygonal  domain in
$\RR^d$, $D^c$ denotes its complement  and 
$$\widetilde \Lambda_s u := \left( (-\Delta)^s\widetilde u\right)|_D$$ 
with $\widetilde u $ denoting the extension of $u$ by 
zero to $\RR^d$.  This fractional Laplacian on
$\RR^d$ is defined using the Fourier transform $\cF$:
\begin{equation}
\cF( (-\Delta)^s f)(\zeta) = |\zeta|^{2s} \cF(f)(\zeta).
\label{Deldef}\end{equation}
The formula \eqref{Deldef} defines an unbounded operator $(-\Delta)^s$
on $L^2(\RRD)$ with domain of definition 
$$D((-\Delta)^s):=\{f\in L^2(\RRD)\ : \
|\zeta|^{2s} \cF ( f) \in L^2(\RRD)\}.$$
It is clear that the Sobolev space
$$
H^{2s}(\RRD) := \left\lbrace f \in L^2(\RRD) \ : \ (1+|\zeta|^2)^{s} \cF(f) \in L^2(\RRD)\right\rbrace
$$
is a subset of $D((-\Delta)^s)$ for any $s\ge 0$.  Note that  $(-\Delta)^sv$ for
$s=1$ and $v\in H^2(\RRD)$ coincides with the negative Laplacian applied
to $v$.

The term $-\widetilde\Lambda_s$ along with the ``boundary condition"
\eqref{eq2} represents the generator of a symmetric $s$-stable 
L\'evy process which is killed when it exits $D$ (cf. \cite{mueller}). The
$f(u)$ term in \eqref{eq1} involves  white noise and will be
ignored in this paper.

The goal of this paper is to study the numerical approximation of solutions of
partial differential equations on bounded domains involving the
fractional operator $\widetilde \Lambda_s$ supplemented with the boundary conditions \eqref{eq2}. As finite element approximations to parabolic
problems are based on approximations to the elliptic part, we shall
restrict our attention to the elliptic case, namely, 
\begin{equation} \begin{aligned}
\widetilde \Lambda_s u &=f, \qquad \text{in } D,\\
u&=0, \qquad \text{in } D^c.\\
\end{aligned}
\label{mainP}
\end{equation}
The above system is sometimes referred to as the \emph{``integral"} fractional Laplacian problem.

We note that the variational formulation of \eqref{mainP} can be defined
in terms of the classical spaces
$\widetilde H^s(D)$ consisting of the functions defined in $D$ whose
extension by zero are in $H^s(\RRD)$.  This is to find $u\in
\widetilde H^s(D)$ satisfying 
\begin{equation}
a(u,\phi) = \int_D f \phi\, dx, \Forall \phi \in \widetilde H^s(D),
\label{mainWP}
\end{equation}
where
\begin{equation}\label{e:variational1}
a(u,\phi) = \int_{\RRD} [(-\Delta)^{s/2} \widetilde u ][(-\Delta)^{s/2}\widetilde\phi]\, dx
\end{equation}
with $\widetilde u$ and $\widetilde \phi$ denoting the extensions by 0.
We refer to Section~\ref{s:analytical} for the description of model problems.
The bilinear form $a(\cdot,\cdot)$ is obviously bounded on $\tHs(D)\times \tHs(D)$ and, as discussed in Section~\ref{s:notation}, it
is coercive on $\tHs(D)$.
Thus,  the Lax-Milgram Theory guarantees existence
and uniqueness.

We consider finite element approximations of \eqref{mainWP}.  The use of
standard finite element approximation spaces of continuous functions
vanishing on $\partial D$ is the natural choice.  The convergence
analysis is classical once the regularity properties of solutions to
problem \eqref{mainWP} are understood (regularity results for
\eqref{mainWP} have been studied  in \cite{AB15} and
\cite{ros-oton}). However, the implementation of the resulting
discretization suffers from the fact that, for $d>1$, the entries of the
stiffness matrix, namely, $a(\phi_i,\phi_j)$, with $\{\phi_k\}$ denoting
the finite element basis, cannot be computed exactly.

When $d=1$, $s\in (0,1/2)\cup (1/2,1)$   and, for example, $D=(-1,1)$, 
the bilinear form can be written
in terms of Riemann-Liouville fractional derivatives (cf. \cite{KST06}), namely,
\begin{equation} a(\phi_i,\phi_j) = \frac{(\partial_L^{s}\phi_i,\partial_R^{s}\phi_j)_D
+(\partial_L^{s}\phi_j,\partial_R^{s}\phi_i)_D}{ 2 \cos(s\pi)} . \label{onedstiff}
\end{equation}
Here $(\cdot,\cdot)_D$ denotes the inner product on $L^2(D)$ and for $t\in (0,1)$ and $v\in H^1_0(D)$,
the left-sided and right-sided Riemann-Liouville fractional derivatives of order $t$ are defined by
\begin{equation}
\partial^{t}_Lv(x) = \frac 1 {\Gamma(1-t)}\frac{d}{dt} \int_{-1}^x \frac{v(y)}{(x-y)^{t} }\, dy
\label{onedL}
\end{equation}
and
\begin{equation}
\partial^{t}_Rv(x) = \frac 1 {\Gamma(1-t)}\frac{d}{dt} \int_x^1 \frac{v(y)}{(x-y)^{t} }\, dy .
\label{onedR}
\end{equation}
Note that the integrals in \eqref{onedL} and \eqref{onedR} can be  easily computed when $v$ is
a piecewise polynomial, i.e, when $v$ is a finite element basis
function. The computation of the stiffness matrix in this case reduces
to a coding exercise.

A representation of the fractional Laplacian for $d\ge 1$ is given by
\cite{ST10}:
\begin{equation}\label{e:PV}
((-\Delta)^\pow \eta)(x) = c_{d,\pow} PV \int_\RRD \frac {\eta(x)-\eta(y)}
{|x-y|^{d+2s}}\,
dy, \Forall \eta\in \mathcal S ,
\end{equation}
where $\mathcal S$ denotes the Schwartz space of rapidly decreasing
functions on $\RRD$,
$PV$ denotes the principle value and $c_{d,\pow}$ is a normalization
constant. It follows that for $\eta,\theta\in \mathcal S$,
\begin{equation} 
a(\eta,\theta)= ((-\Delta)^\pow \eta,\theta)= 
\frac {c_{d,\pow}}{2} \int_{\RRD}\int_{\RRD}
	\frac{(\eta(x)-\eta(y))(\theta(x)-\theta(y))}{|x-y|^{d+2\pow}}
	\, dy\, dx.
\label{schwpair}
\end{equation}
A density argument implies that the stiffness entries are  given by
\begin{equation}\label{stiff_int}
	a(\phi_i,\phi_j)=
\frac {c_{d,\pow}}{2} \int_{\RRD}\int_{\RRD}
	\frac{(\widetilde \phi_i(x)-\widetilde\phi_i(y))(\widetilde\phi_j(x)-\widetilde\phi_j(y))}{|x-y|^{d+2\pow}}
	\, dy\, dx,
	\end{equation}
	where again $\widetilde \phi$ denotes the extension of $\phi$ by zero outside $D$.
It is possible to apply the techniques developed for the approximation
of boundary integral stiffness matrices \cite{BEM} to deal with some of the issues
associated with the approximation of the double integral above,  
 namely, the
application of special techniques for handling the singularity and
quadratures. However, \eqref{stiff_int} requires additional
truncation techniques as the non-locality of the kernel implies a
non-vanishing integrand over $\RRD$.    These techniques are used to 
approximate \eqref{stiff_int} in \cite{DG13,AB15}.  In particular, 
\cite{AB15} use their regularity theory to do {\it a priori} mesh
refinement near the boundary to develop higher order convergence
under the assumption of exact evaluation of the stiffness matrix.

The method to be developed in this paper is based on a 
representation of the underlying bilinear form given in
Section~\ref{s:alternative}, namely,  for $s\in (0,1)$, $0\le r \le s$, 
$\eta\in H^r(\RRD)$ and $\theta\in H^{s-r}(\RRD)$,
\begin{equation}
\int_\RRD [(-\Delta)^{r/2}\eta][(-\Delta)^{(s-r)/2} \theta] \, dx 
 = c_s \int_0^\infty t^{2-2s} ((-\Delta)(I-t^2\Delta)^{-1}
\eta,\theta)\, \frac {dt}t \label{intid}
\end{equation}
where $(\cdot,\cdot)$ denotes the inner product on $L^2(\RRD)$ (see,
also \cite{bacuta-thesis}).  We note
that for $t>0$, $( I-t^2\Delta)^{-1}$ is a bounded map of $L^2(\RRD)$
into $H^2(\RRD)$ so that the integrand above is well defined for
$\eta,\theta\in L^2(\RRD)$.  In Theorem~\ref{t:alternate1}, we show
that for $\eta\in H^r(\RRD)$ and
$\theta\in H^{s-r}(\RRD)$, the formula \eqref{intid} holds and the right hand side integral converges
absolutely.
It follows that the bilinear form $a(\cdot,\cdot)$ is given 
by
\begin{equation}
a(\eta,\theta)
 = c_s \int_0^\infty t^{2-2s} ((-\Delta)(I-t^2\Delta)^{-1}
\widetilde \eta,\theta)_D\, \frac {dt} t,\Forall \eta,\theta\in \tHs(D) .
\label{intidD}
\end{equation}

There are three main issues needed to be addressed in developing numerical
methods for \eqref{mainWP} based on \eqref{intidD}:
\begin{enumerate}[(a)]
\item  The infinite integral with respect to $t$ must be truncated and
  approximated by numerical quadrature;
\item At each quadrature node $t_j$, the inner product term in the integrand involves an elliptic problem
  on $\RRD$.
  This must be replaced by a problem with vanishing boundary condition on a bounded truncated
  domain $\Omega^M(t_j)$ (defined below);
\item Using a fixed subdivision of $D$, we construct subdivisions of the larger domain $\Omega^M(t_j)$ which
  coincide with that on $D$.  We then replace the problems on
  $\Omega^M(t_j)$ of (b) above by their finite element approximations.
\end{enumerate}

We address (a) above by first making the change of variable $t^{-2}=e^y$
which results in an integral over $\RR$.   We then apply a sinc
quadrature obtaining  the approximate bilinear form
\begin{equation}
a^k(\eta,\theta):=\frac {c_s k} 2 \sum_{j=-N^-}^{N^+} e^{sy_j} 
((-\Delta) (e^{y_j}I-\Delta)^{-1} \widetilde \eta,\theta)_D,\Forall
\theta,\eta\in L^2(D),
\label{quad}
\end{equation}
where $k$ is the quadrature spacing, $y_j=kj$, and $N^-$ and $N^+$ are
positive integers.  Theorem~\ref{t:quad_error} shows that for 
$\theta\in\tHs(D)$ and $\eta\in \widetilde{H}^\delta(D)$ with
$\delta\in(s,2-s]$, we have
$$
\begin{aligned}
&|a(\eta,\theta)-a^k(\eta,\theta)| \\
&\quad\le C(\delta,s, \dd)\big[e^{-2\pi\dd/k}
+e^{\mathbin{(s-\delta)}N^+k/2}+e^{-s kN^-}\big]
\|\eta\|_{\widetilde{H}^\delta(D)}
\|\theta\|_{\widetilde{H}^s(D)},
\end{aligned}
$$
where $0<\dd<\pi$ is a fixed constant.
Balancing of the exponentials gives rise to an
$O(e^{-2\pi\dd/k})$ convergence rate with the relation $\Nplus+\Nminus=O(1/k^2)$.

The size of the truncated domain $\Omega^M(t_j)$ in (b) is determined by the decay of $(e^{y_j}I-\Delta)^{-1} f$ for functions $f$ supported in $D$.   For technical
reasons, we first extend $D$ to a bounded convex  (star-shaped with
respect to the origin)  domain
$\Omega$ and set (with $t_j=e^{-y_j/2}$)
$$
\Omega^M(t_j):= \left\lbrace
\begin{array}{ll}
\left\lbrace (1+t_j(1+M))x \ : \ x \in \Omega \right\rbrace, &\qquad t_j\geq 1\\
\left\lbrace (2+M)x \ : \ x \in \Omega \right\rbrace, &\qquad t_j < 1.
\end{array}\right. 
$$
Let $\Delta_t$ denote the 
unbounded operator  on $L^2(\Omega^M(t))$ corresponding to the Laplacian
on $\Omega^M(t)$ supplemented with vanishing boundary condition.   We define 
the bilinear form $a^{k,M}(\eta,\theta)$ by replacing 
$(-\Delta)(e^{y_j}I-\Delta)^{-1}$ in \eqref{quad} by $(-\Delta_{t_j}) (e^{y_j}I
-\Delta_{t_j})^{-1}$.  Theorem~\ref{t:truncate} guarantees that for
sufficiently large $M$, we have
$$|a^k(\eta,\theta)-a^{k,M}(\eta,\theta)|\le
Ce^{-cM}\|\eta\|_{L^2(D)}\|\theta\|_{L^2(D)},\quad\Forall \eta,\theta\in
L^2(D).$$
Here $c$ and $C$ are positive constants independent of $M$ and $k$.  This addresses (b). 

Step (c) consists in approximating  $(-\Delta_{t_j}) (e^{y_j}I 
-\Delta_{t_j})^{-1}$ using finite elements.
In this aim, we associate to a subdivision of $\Omega^M(t_j)$ the finite element space $\vh^M(t_j)$ and the restriction  $a^{k,M}_h(\cdot,\cdot)$ of $a^{k,M}(\cdot,\cdot)$ to  $\vh^M(t_j) \times \vh^M(t_j)$.
As already mentioned, the subdivisions of $\Omega^M(t_j)$ are constructed to coincide on $D$.
Denoting by $\vh(D)$ the set of finite element functions restricted to $D$ and vanishing on 
$\partial D$, our approximation to the solution of \eqref{mainWP} is the function $u_h\in \vh(D)$ satisfying 
\begin{equation}
a_h^{k,M}(u_h,\theta) = \int_D f \theta\, dx,\Forall \theta\in \vh(D).
\label{fem}
\end{equation}  
Lemma~\ref{l:ellipticity} guarantees the $\vh(D)$-coercivity of the bilinear form $a^{k,M}_h(\cdot,\cdot)$. Consequently, $u_h$ is well defined again from the Lax-Milgram theory. 
Moreover, given, for every $t_j$,  a sequence of quasi-uniform subdivisions of $\Omega^M(t_j)$, we show (Theorem~\ref{t:discrete_consistency}) that for $v$ in $\widetilde H^\beta(D)$ with $\beta\in (s,3/2)$ and for $\theta_h\in \vh(D)$,
$$	
	|a^{k,M}(v_h,\theta_h)-a^{k,M}_h(v_h,\theta_h)|\le C(1+\ln (h^{-1})) h^{\beta-s}
	\|v\|_{\widetilde H^ \beta(D)}\|\theta_h\|_{\widetilde H^s(D)}.
$$
Here $C$ is a constant independent of $M,k$ and $h$, and $v_h\in \vh(D)$ denotes the Scott-Zhang interpolation or the $L^2$ projection of $v$ depending on whether $\beta \in (1,3/2)$ or $\beta \in (s,1]$.

Strang's Lemma implies that the 
error between $u$ and $u_h$ in the $\tHs(D)$-norm is bounded by the error of the best approximation in $\tHs(D)$ and 
the sum of the consistency errors from the above three steps (see Theorem~\ref{t:total_error}).

The online of the paper is as follows. Section~\ref{s:notation} introduces notations of Sobolev spaces followed by Section~\ref{s:dotted_space} introducing the dotted spaces associated with elliptic operators. 
The alternative 
integral representation of the bilinear form is given in Section~\ref{s:alternative}. Based on this integral representation,
we discuss the discretization of the bilinear form and the associated consistency error in three steps (Sections~\ref{s:quad}, \ref{s:truncated} and \ref{s:FEM}). The energy error estimate for the discrete problem is given in Section~\ref{s:FEM}.
A discussion on the implementation aspects  of the method together with results of numerical experiments
illustrating the  convergence of the method are provided in Section~\ref{s:numerical}. 
We left to Appendix
the proof of technical result regarding the stability and approximability of the Scott-Zhang interpolant in nonstandard norms. 

\section{Notations and Preliminaries}\label{s:notation}

\subsubsection*{Notation}

We use the notation $D\subset \RRD$ to denote the polygonal domain with Lipschitz boundary in problem \eqref{mainWP} and
$\omega \subset \RRD$ to denote a generic  bounded Lipschitz domain.
For a function $\eta: \omega \rightarrow \RR$, we denote by $\tilde \eta$ its extension by zero outside $\omega$.
We do not specify the domain $\omega$ in the notation $\tilde \eta$ as it will be always clear from the context.

\subsubsection*{Scalar Products}

We denote by $(\cdot,\cdot)_{\omega}$ the $L^2(\omega)$-scalar product and 
by $\|\cdot\|_{L^2(\omega)}:= (\cdot,\cdot)_{\omega}^{1/2}$ the associated norm. 
The $L^2(\RRD)$-scalar product is denoted $(\cdot,\cdot)_\RRD$.
To simplify the notation, we write in short $(\cdot,\cdot):=(\cdot,\cdot)_\RRD$ and $\|\cdot\|:=\|\cdot\|_{L^2(\RRD)}$.  

\subsubsection*{Sobolev Spaces.}
For $r>0$, the
Sobolev space of order $r$ on $\RRD$, $H^r(\RRD)$, is defined to
be the set of functions $\theta\in L^2(\RRD)$ such that 
\begin{equation}
\|\theta\|_{H^r(\RRD)} := \bigg(\int_\RRD (1+|\zeta|^2)^{r/2}
|\mathcal F(\theta)(\zeta)|^2\, d\zeta\bigg)^{1/2} < \infty.
\label{halpharrd}
\end{equation}

In the case of bounded Lipschitz domains,  $H^r(\omega)$ with $r \in (0,1)$, stands for the Sobolev space of
 order $r$ on $\omega$. It is equipped with the Sobolev--Slobodeckij norm, i.e.
 \begin{equation} 
\|\theta\|_{H^r(\omega)}:= \big(\|\theta\|_{L^2(\omega)}^2 +
 |\theta|_{H^r(\omega)}^2\big)^{1/2} ,
\label{hbetanorm}
\end{equation}
where
$$ |\theta|_{H^r(\omega)}^2:=\int_\omega \int_\omega \frac {(\theta(x)-\theta(y))^2}
{|x-y|^{d+2r}} \, dx\,dy.$$ 
When $r\in (1,2)$ instead, the norm in 
$H^r(\omega)$ is given by
$$\|\theta\|_{H^r(\omega)}^2 := \|\theta\|_{H^1(\omega)}^2 +
\int_{\omega} \int_{\omega} \frac
{|\nabla \theta(x)-\nabla \theta(y)|^2} {|x-y|^{d+2(r-1)}} \, dx\, dy ,$$
where $\|w\|_{H^1(\omega)}:=(\|w\|_{L^2(\omega)}^2+\||\nabla w|\|_{L^2(\omega)}^2)^{1/2}$.
In addition, $H^1_0(\omega)$ denotes  the set of functions in $H^1(\omega)$ vanishing at $\partial \omega$, the boundary 
of $\omega$. 
Its dual space is denoted $H^{-1}(\omega)$.
We note that when we replace $\omega$ with $\RRD$ and $r\in[0,2)$, the norms using the double integral 
above are equivalent with those in \eqref{halpharrd} (see e.g. \cite{lions,mclean}).

\subsubsection*{The spaces $\tHr(D)$.}
For $r\in (0,2)$, the set of functions in $D$ whose extension by zero are in $H^s(\RRD)$ is denoted $\widetilde H^r(D)$.
The norm of $\tHr(D)$ is given by $\|\tilde{\cdot}\|_{H^r(\RRD)}$.
Note that for $r\in(0,1)$, \eqref{e:PV} implies that for $\phi$ in the Schwartz space $\cS$,
\begin{equation} 
((-\Delta)^r \phi,\phi)=
{|c_{d,r}|} \int_{\RRD}\int_\RRD \frac {(\phi(x)-\phi(y))^2}
{|x-y|^{d+2r}}\,
dx\, dy.
\label{tartar}
\end{equation}
Thus, we prefer to use
\begin{equation}\label{d:Hsnorm}
\| \phi \|_{\tH^r(D)}:= \bigg( {|c_{d,r}|} \int_{\RRD}\int_\RRD \frac
{(\widetilde \phi(x)-\widetilde \phi(y))^2}
{|x-y|^{d+2r}}\,
dx\, dy \bigg)^{1/2}
\end{equation}
as equivalent norm on $\tH^r(D)$ for $r\in(0,1)$.
This is justified upon invoking a variant of the
Peetre-Tartar compactness 
argument on $\tHr(D) \subset H^r(D)$.

\subsubsection*{Coercivity.}
Since $C_0^\infty(D) $ is dense in $\tHs(D)$ for $s\in (0,1)$ \cite{grisvard},  \eqref{schwpair} and a
density  argument imply that for 
$\eta,\theta\in \tHs(D)$, we have 
$$a(\eta,\theta) =\frac {c_{d,s}}{2} \int_{\RRD}\int_\RRD
	\frac{(\widetilde \eta(x)-\widetilde \eta(y))(\widetilde \theta(x)-\widetilde\theta(y))}{|x-y|^{d+2s}}
	\, dy\, dx.
$$
In turn, from the definition \eqref{d:Hsnorm} of the $\tH^s(D)$ norm, we directly deduce the coercivity of $a(\cdot,\cdot)$ on
$\tHs(D)$
\begin{equation}\label{e:coercivity_a}
a(\eta,\eta) = \| \eta \|^2_{\tH^s(D)}, \qquad \forall \eta \in \tH^s(D). 
\end{equation}

\subsubsection*{Dirichlet Forms}

We define the Dirichlet form on $H^1(\omega)\times H^1(\omega)$ to be
$$
d_{\omega}(\eta,\phi) := \int_{\omega} \nabla \eta \cdot \nabla \phi\, dx.
$$
On $H^1(\RRD)\times H^1(\RRD)$ we write
$$
d(\eta,\phi):=d_{\RRD}(\eta,\phi) := \int_{\RRD} \nabla \eta \cdot \nabla \phi\, dx.
$$

\section{Scales of interpolation spaces}\label{s:dotted_space}

We now introduce another set of functions instrumental in the analysis of the finite element method described in Section~\ref{s:truncated} and \ref{s:FEM}.
In this section $\omega$ stands for a bounded domain of $\RRD$.

Given
$f\in L^2(\omega)$, we define $\theta \in H^1_0(\omega)$ to be the unique
solution to 
\begin{equation}
(\theta,\phi)_{\omega} + d_{\omega}(\theta,\phi) =
(f,\phi)_{{\omega}},
\Forall \phi \in H^1_0({\omega})
\label{varscale}
\end{equation}
and define $T_{\omega} : L^2({\omega})\rightarrow H^1_0({\omega}) $ by 
\begin{equation}\label{e:Tomega}
T_\omega f=\theta.
\end{equation}
As discussed in \cite{kato1961}, this defines a densely
defined  unbounded
operator on $L^2({\omega})$, namely $L_{\omega} f:=T_{\omega}^{-1} f$ for $f$
in
$$D(L_{\omega}) := \{ T_{\omega}  \phi\ : \ \phi\in L^2({\omega})\}.$$
 The
operator $L_{\omega}$ is self-adjoint and positive so its fractional powers 
define a Hilbert
scale of interpolation spaces,  namely, for $r\ge 0$,
$$\dH r({\omega}) := D(L_{\omega}^{r/2})$$ 
with $D(L_{\omega}^r)$ denoting the domain of $L_{\omega}^r$.    These are Hilbert spaces
with norms 
$$\|w\|_{\dH r({\omega})} := \|L_{\omega}^{r/2} w\|_{L^2({\omega})}.$$
The space $\dH 1({\omega})$ coincides with $H^1_0({\omega})$ while $\dH 0(\omega)$ with $L^2(\omega)$, in both cases with equal norms.
Hence for $r\in[0,1]$, we have
$$\dH r({\omega})= (L^2({\omega}),H^1_0({\omega}))_{r,2},$$
where $(L^2({\omega}),H^1_0({\omega}))_{r,2}$ denotes the interpolation
spaces defined using the real method.  

Another characterization  of these spaces stems from Corollary 4.10 in \cite{chandler}, which states that for $r\in [0,1]$, 
the spaces $\tHr(\omega)$ are interpolation spaces.
Since $\tH^1(\omega)=H^1_0(\omega)$ and 
$\tH^0(\omega)=L^2(\omega)$,  $\tHs(\omega)$ coincides with
$\dH r(\omega)$. In particular, we have
\begin{equation}\label{e:dot-tilde}
C^{-1} \|\theta\|_{\dH {r}(\omega)}\le  \|\theta\|_{\widetilde H^{r}(\omega)} \le C  \|\theta\|_{\dH {r}(\omega)},
\end{equation}
for a constant $C$ only depending on $\omega$. 

The intermediate spaces can also be characterized by expansions in
the $L^2({\omega})$ orthonormal system of eigenvectors $\{\psi_i\}$    for
$T_{\omega}$, i.e., 
$$\dH {r}({\omega})= \bigg \{ \phi
\in L^2({\omega})\ : \sum_{i=1}^\infty \lambda_i
^r |(\phi,\psi_i)_{\omega} |^2 <\infty\bigg\}.$$
Here $\lambda_i=\mu_i^{-1}$ where $\mu_i$ is the eigenvalue of
$T_{\omega}$ associated with $\psi_i$.
In this case, we find that 
$$
\|\phi\|^2_{\dH r({\omega})} =\|L_{\omega}^{r/2}\phi\|^2_{L^2({\omega})} =\sum_{i=1}^\infty \lambda_i
^r |(\phi,\psi_i)_{\omega} |^2$$
and for $r\in (0,1)$, (see, e.g., \cite{00BZ}) 
$$\|\phi\|^2_{\dH r({\omega})}=\frac{2\sin\pi r}\pi \int_0^\infty t^{-2r} K_\omega(\phi,t) \frac
{dt}t.
$$
Here 
$$K_\omega(\phi,t) :=\inf _{w\in H^1_0({\omega}) } (\|\phi-w\|_{L^2({\omega})}^2+t^2
\|w\|_{H^1({\omega})}^2).
$$
Note that if ${\omega}' \subset {\omega}$ then since the extension of a
function $\phi$ in $H^1_0({\omega}')$ by zero is in $H^1_0({\omega})$,  the K-functional
identity implies that for all $r\in [0,1]$,
\begin{equation}
 \|\widetilde \phi\|_{\dH r({\omega})}\le \|\phi\|_{\dH r({\omega}')},
\label{extenhs}
\end{equation}
where $\widetilde \phi$ denotes the extension by zero of $\phi$ outside $\omega'$.

The operator $T_{\omega}$ extends naturally to $F\in H^{-1}({\omega})$ by setting
$T_{\omega} F=u$ where $u\in H^1_0({\omega})$ is the solution of
\eqref{varscale} with $(f,\phi)_{\omega}$ replaced by $\langle F,\phi
\rangle $. Here $\langle \cdot,\cdot\rangle $ denotes the
functional-function pairing.  
Identifying $f\in L^2({\omega})$ with the
functional $\langle F,\phi \rangle := (f,\phi)_{\omega}$, we define the
intermediate spaces for $r\in (-1,0) $ by
$$\dH {r}({\omega}) := (H^{-1}({\omega}),L^2({\omega}))_{1+r,2}$$
and set $\dH {-1}:= H^{-1}(\omega)$.
Since $T_{\omega}$ maps $H^{-1}({\omega})$ isomorphically onto $\dH 1({\omega})$ and
$L^2({\omega})$ isomorphically onto $\dH 2({\omega})$, 
$T_{\omega}$ maps $\dH {-r}({\omega})$ isometrically onto $\dH {2-r}(\omega)$ for
$r\in [0,1]$. 

Functionals in $H^{-1}({\omega})$ can also be characterized in terms of
the eigenfunctions of $T_{\omega}$, indeed, 
$H^{-1}({\omega})$ is the set of linear functionals $F$ for which the sum
$$\sum_{i=1}^\infty \lambda_i^{-1}  |\langle F,\psi_i \rangle |^2 $$
is finite.   Moreover,
$$\|F\|_{H^{-1}({\omega})} = \sup_{\theta\in H^1_0({\omega})} \frac {\langle
  F,\theta \rangle} {\|\theta\|_{H^1({\omega})}}= \bigg
(\sum_{i=1}^\infty \lambda_i^{-1}  |\langle F,\psi_i \rangle |^2 \bigg
)^{1/2}$$
for all $ F\in H^{-1}({\omega})$.
This implies that for $r\in [-1,0]$,
$$\dH r({\omega}) = \{ F\in \dH {-1}\ : \ \sum_{i=1}^\infty \lambda_i^{r}
|\langle F,\psi_i \rangle |^2<\infty\}$$
and 
$$  \|F\|_{\dH r(\omega)} = \bigg(\sum_{i=1}^\infty \lambda_i^{r}
|\langle F,\psi_i \rangle |^2\bigg)^{1/2}.$$

\begin{remark}[Norm equivalence for Lipschitz domains]\label{r:beta-equiv}
For $r\in(1,3/2)$, it is known that $\tHr(\omega)=H^r(\omega)\cap H^1_0(\omega)$. 
On the other hand, we note that when $\partial\omega$ 
is Lipschitz, $-\Delta$ is an isomorphism 
from $H^r(\omega)\cap H^1_0(\omega)$ to $\dot H^{r-2}(\omega)$; see Theorem 0.5(b) of \cite{JK95}. 
We apply this regularity result into Proposition 4.1 of \cite{BP13}
to obtain $H^r(\omega)\cap H^1_0(\omega)=\dot H^r(\omega)$. So
the norms of $\tHr(\omega)$ and $\dot H^r(\omega)$ are equivalent for $r\in[0,3/2)$ 
and the equivalence constant may depend on $\omega$. In what follows,
we use $\tHr(D)$ to describe the smoothness of functions defined on $D$. 
When functions defined on a larger domain (see Section~\ref{s:truncated} and \ref{s:FEM}),
we will use these interpolation spaces separately so that we can investigate 
the dependency of constants.
\end{remark}

We end the section with the following lemma:

\begin{lemma} \label{Tshift} 
Let $a$ be in $[0,2]$ and $b$ be in $[0,1]$
  with $a+b\le 2$.   Then for
   $\mu\in (0,\infty)$, we have
$$\| (\mu I + T_{\omega})^{-1} \phi \|_{\dH {-b}(\omega) } \le  \mu^{(a+b)/2-1}
\|\phi \|_{\dH a({\omega})},\Forall \phi\in \dH a({\omega}).$$
\end{lemma}

\begin{proof}  Let $\phi$ be in $\dH a({\omega})=D(L_\omega^{a/2})$.
Setting $\theta:=L_{\omega}^{a/2} \phi \in L^2(\omega)$, it suffices to prove that
 \begin{equation}
\| (\mu I + T_{\omega})^{-1} T_{\omega}^{a/2} \theta \|_{\dH {-b}(\omega) } \le   \mu^{(a+b)/2-1}
\|\theta \|_{L^2({\omega})},\Forall \theta\in L^2(\omega).\label{rewrite}
\end{equation}
The operator $T_{\omega}$ and its fractional powers are symmetric in
the $L^2({\omega})$ inner product.
Therefore, we have
$$\begin{aligned} \| (\mu I + T_{\omega})^{-1} T_{\omega}^{a/2} \theta \|^2_{\dH {-b}(\omega) } &=
\sum_{i=1}^\infty  |( (\mu I + T_{\omega})^{-1} T_{\omega}^{a/2}
\theta,\psi_i)_{\omega}|^2 \lambda_i^{-b}\\
&=\sum_{i=1}^\infty \frac {\lambda_i^{-a-b}} {(\mu +
  \lambda_i^{-1})^2} |(\theta,\psi_i)_{\omega}|^2.
\end{aligned}
$$
Inequality \eqref{rewrite} follows from Young's inequality
$$\lambda_i^{-(a+b)/2} \mu^{1-(a+b)/2} {(\mu +
  \lambda_i^{-1})}\le 1.
$$
\end{proof}

\section{An Alternative Integral Representation of the Bilinear Form}\label{s:alternative}

The goal of this section is to derive the integral expression
\eqref{intid} and some of its properties.  

\begin{theorem}[Equivalent Representation] \label{t:alternate1} 
Let $s\in(0,1)$ and  $0\le r\le s$.   For $\eta\in H^{s+r}(\RRD)$ and
$\theta\in H^{s-r} (\RR^d)$,
\begin{equation}
((-\Delta)^{(s+r)/2} \eta,(-\Delta)^{(s-r)/2}\theta)= c_s \int _0^\infty t^{2-2s} (-\Delta (I-t^2\Delta)^{-1}\eta,\theta)\, 
\frac {dt}t,
\label{is0}
\end{equation}
where
\begin{equation}\label{e:cs}
c_s:=\bigg( \int_0^\infty \frac {y^{1-2s}} {1+y^2}\, dy\bigg)^{-1} 
=\frac{2\sin(\pi s)}{\pi}.
\end{equation}
\end{theorem}

\begin{proof}  
Let $I(\eta,\theta)$ denotes the right hand side of \eqref{is0}.    
Parseval's theorem implies that
\begin{equation}
 (-\Delta (I-t^2\Delta)^{-1}\eta,\theta)=
\int_\RRD \frac {|\zeta|^2}{1+t^2 |\zeta|^2} \cF(
\eta)(\zeta)\overline{\cF(
\theta)(\zeta)}  \, d\zeta.
\label{ftid}
\end{equation}
and so
\begin{equation}    I(\eta,\theta) 
=c_s \int _0^\infty t^{1-2s} 
\int_\RRD \frac {|\zeta|^2}{1+t^2 |\zeta|^2} \cF(
\eta)(\zeta) \overline{\cF(
\theta)(\zeta)}\, d\zeta\, {dt}.
\label{firsti}
\end{equation}
In order to invoke Fubini's theorem, we now show that
$$
c_s
\int_\RRD  \int _0^\infty t^{1-2s} \frac {|\zeta|^2}{1+t^2 |\zeta|^2}  |\cF(
\eta)(\zeta)|\ | \cF(
\theta)(\zeta)|\, d\zeta\, {dt} < \infty.
$$
Indeed, the change of variable $y=t|\zeta|$ and the definition \eqref{e:cs} of $c_s$ implies that the above integral is equal to
\begin{equation*}
\begin{split}
c_s\int_\RRD |\cF(
\eta)(\zeta)||\cF(
\theta)(\zeta)|  \int _0^\infty t^{1-2s}
 \frac {|\zeta|^2}{1+t^2 |\zeta|^2} \, dt\, d\zeta =  \int_\RRD|\zeta |^{2s}  |\cF(
\eta)(\zeta)|\,  |\cF(
\theta)(\zeta)| \, \, d\zeta,  
\end{split}
\end{equation*}
which is finite for $\eta\in H^r(\RRD)$ and
$\theta\in H^{s-r} (\RR^d)$.
We now apply Fubini's theorem and  the same change of variable $y=t|\zeta|$ in \eqref{firsti} to arrive at
$$\begin{aligned}  I(\eta,\theta) 
&=\int_\RRD  |\zeta| ^{2s} \cF(
\eta)(\zeta) \overline{\cF(
\theta)(\zeta)} \, d\zeta= 
((-\Delta)^{(s+r)/2} \eta,(-\Delta)^{(s-r)/2}\theta).
\end{aligned}$$
This completes the proof.\end{proof}

Theorem~\ref{t:alternate1} above implies that  for $\eta,\theta$ in $\tHs(D)$,
\begin{equation}\begin{aligned} a(\eta,\theta) &=  c_s \int _0^\infty t^{-2s} (w(\widetilde\eta,t),\theta)_D\, 
\frac {dt}t,
\end{aligned}
\label{is}
\end{equation}
where for $\psi \in L^2(\RRD)$
$$w(t):=w(\psi,t):= -t^2  \Delta
(I-t^2\Delta)^{-1}\psi .$$
Examining the Fourier transform of $w(\psi,t)$, we realize that
$w(t):=w(\psi,t):=\psi+v(\psi,t)$
where $v(t):=v(\psi,t) \in H^1(\RRD) $ solves
\begin{equation}
(v(t),\phi)+t^2 d(v(t),\phi)=-(\psi,\phi),\Forall \phi \in
H^1(\RRD) .
\label{vt}
\end{equation}

The integral in \eqref{is} is the basis of a numerical method for
\eqref{mainWP}.  The following lemma, instrumental in our analyze, provides an alternative
characterization for the inner product appearing on the right hand side
of \eqref{is}.

\begin{lemma}\label{l:K} Let $\eta$ be in $L^2(\RRD)$. Then,
\begin{equation}
(w(\eta,t),\eta)= \inf_{\theta\in H^1(\RRD) } \{ \|\eta-\theta\|^2 +t^2
d(\theta,\theta)\}=: K(\eta,t).
\label{kmin1}
\end{equation} 
\end{lemma}

\begin{proof}
Let $\eta$ be in $L^2(\RRD)$.
We start by observing that for any positive $t$ and $\zeta\in \RRD$,
$$\hat \phi(\zeta):= \frac{ \cF(\eta) (\zeta)} {1+t^2|\zeta|^2} $$
solves the minimization problem
$$\inf_{z\in \CC} \{|\cF(\eta)(\zeta)-z|^2+t^2 |\zeta|^2 |z|^2\}$$
and so
\begin{equation}
\inf_{z\in \CC} \{|\cF(\eta)(\zeta)-z|^2+t^2 |\zeta|^2 |z|^2\}=
\frac {t^2 |\zeta|^2} {1+t^2|\zeta|^2} |\cF(\eta)(\zeta)|^2.
\label{point}
\end{equation}

We
denote $\phi$ 
to be the inverse Fourier transform of $\hat \phi$. 
Note that $\phi$ is in 
$H^1(\RRD)$  (actually,  $\phi$ is in $ H^2(\RRD)$).  

Applying the Fourier transform, we find that 
\begin{equation}
 K(\eta,t)=\inf_{\theta\in H^1(\RRD) } \int_\RRD 
(|\cF(\eta)(\zeta)- \cF(\theta)(\zeta)|^2 +t^2 |\zeta|^2 |\cF(\theta)(\zeta)|^2)\,
d\zeta.
\label{kmin2}
\end{equation}
Now, $\phi$ is the pointwise minimizer of the integrand in
\eqref{kmin2} and since $\phi \in H^1(\RRD)$, it is also the minimizer of \eqref{kmin1}.
In addition, \eqref{point}, \eqref{kmin2} and \eqref{ftid}
 imply that
$$ K(\eta,t)= \int_\RRD \frac {t^2 |\zeta|^2} {1+t^2|\zeta|^2} |\cF(\eta)(\zeta)|^2
\, d\zeta = (w(\eta,t),\eta).
$$
This completes the proof of the lemma.
\end{proof}

\begin{remark}[Relation with the vanishing Dirichlet boundary condition case] \label{r:bounded_K}
The above lemma implies that for $\eta\in \tHs(D)$,
$$ a(\eta,\eta) 
= c_s \int _0^\infty t^{-2s} K(\widetilde \eta,t)\, 
\frac {dt}t.
$$
It is observed in the Appendix of \cite{00BZ} that for any bounded domain
$\omega$, and $\eta\in
(L^2(\omega),H^1_0(\omega))_{s,2}$, the real interpolation space between $L^2(\omega)$ and $H^1_0(\omega)$, we have
$$\|\eta \|_{(L^2(\omega),H^1_0(\omega))_{s,2}}^2= c_s \int
_0^\infty t^{-2s} K_{\omega}^0 (\eta,t)\, 
\frac {dt}t$$
where
\begin{equation}
K_{\omega}^0(\eta,t) :=\inf_{\theta\in H_0^1(\omega) } \{
  \|\eta-\theta\|_{L^2(\omega)}^2 +t^2
d_{\omega}(\theta,\theta)\}.
\label{infOmega}
\end{equation}

Let $\{\psi_i^0\}\subset H^1_0(\omega)$ denote the $L^2(\omega)$-orthonormal basis of eigenfunctions satisfying  
$$d_{\omega}(\psi_i^0,\theta)= \lambda_i (\psi_i^0,\theta)_{\omega},\Forall \theta\in
H^1_0(\omega).$$
As the proof in Lemma~\ref{l:K} but using the expansion in the above
eigenfunctions,
it is not hard to see that
\begin{equation}
(w_{{\omega}}(\eta,t),\eta)_{\omega} = K_{\omega}^0(\eta,t) 
\label{womegaid}
\end{equation}
with $w_{{\omega}}(\eta,t)=\eta+v$ and $v
\in H^1_0({\omega})$ solving 
$$(v,\theta)_{\omega} +t^2 d_{\omega}(v,\theta)=-(u,\theta)_{\omega},\Forall
\theta \in H^1_0({\omega}).$$
This means that if $\eta\in L^2(\omega)$, $K(\widetilde\eta,t)\le K_\omega^0(\eta,t)$ and hence
$$
(w(\widetilde\eta,t),\eta)_\omega\le (w_{{\omega}}(\eta,t),\eta)_{\omega}.$$
\end{remark}


\section{Exponentially Convergent Sinc Quadrature}\label{s:quad}
In this section, we analyze a sinc quadrature scheme applied to the integral \eqref{is}.
Notice that the analysis provided in \cite{BLP18} does not strictly apply in the present context.

\subsection{The Quadrature Scheme}
We first use the change of variable $t^{-2} = e^y$ so that \eqref{is} becomes
$$
a(\eta,\theta)= \frac{c_s}{2} \int _{-\infty}^\infty e^{sy} (w(\widetilde \eta,t(y)),\theta)_D \, dy.
$$
Given a quadrature spacing $k>0$ and two positive integers $\Nminus$ and $\Nplus$, set
$y_j:=j k$ so that
\begin{equation}
\label{e:tj}
t_j = e^{-y_j/2} = e^{-jk/2}
\end{equation}
and define the approximation of $a(\eta,\theta)$ by
\begin{equation}
a^k(\eta,\theta):=\frac{c_s k}{2}  \sum_{j=-\Nminus}^\Nplus e^{sy_j} (w(\widetilde \eta,t_j),\theta)_D.
\label{ik}
\end{equation}

\subsection{Consistency Bound}
The convergence of the sinc quadrature depends on the properties of the integrand
\begin{equation}
g(y;\eta,\theta):=e^{s y}  (w(\widetilde \eta,t(y)),\theta)_D=e^{sy}\left( -\Delta(e^yI-\Delta)^{-1}\widetilde \eta,\widetilde\theta \right) .
\label{integrant}
\end{equation}
More precisely, the following conditions are required:
\begin{enumerate}[(a)]
 \item $g(\cdot;\eta,\theta)$ is an analytic function in the band
 $$
B=B(\dd):=\left\{z=y+iw \in \mathbb C : \ |w|< \dd\right\} ,$$
where $\dd$ is a fixed constant in $(0, \pi)$.
\item There exists a constant $C$ independent of $y\in\RR$ such that
$$
\int_{-\dd}^{\dd} |g(y+iw;\eta,\theta)|\, dw\leq C;
$$
\item 
$$N(B):=\int_{-\infty}^\infty \left(|g(y+i\dd;\eta,\theta)|+|g(y-i\dd;\eta,\theta)| \right) dy < \infty .$$
\end{enumerate}
In that case, there holds (see Theorem 2.20 of \cite{sinc_int})
\begin{equation}
\label{inf}
\bigg|\int_{-\infty}^\infty g(y;\eta,\theta)\, dy-k\sum_{j=-\infty}^\infty g(kj;\eta,\theta)\bigg|\le \frac{N(B)}{e^{2\pi \dd/k}-1}.
\end{equation}

In our context, this leads to the following estimates for the sinc quadrature error.
\begin{theorem}[Sinc quadrature]\label{t:quad_error}
Suppose $\theta \in\tHs(D)$ and $\eta \in \widetilde{H}^\delta(D)$ with $\delta\in (s,2-s]$.
Let $a(\cdot,\cdot)$ and $a^k(\cdot,\cdot)$ be defined by \eqref{mainWP}  and \eqref{ik}, respectively. Then
we have
\begin{equation}
\begin{aligned}
|a(\eta,\theta)-a^k(\eta,\theta)|
&\le \frac{2c(\dd)}{\delta-s} \left( \frac{2}{e^{2\pi\dd/k}-1} + e^{(s-\delta)\Nplus k/2}\right)\|\eta\|_{\widetilde{H}^\delta(D)} \|\theta\|_{\widetilde{H}^s(D)} \\
& + \frac{c(\dd)}{s} \left( \frac{2}{e^{2\pi\dd/k}-1} + e^{-s \Nminus k}\right) \|\eta \|_{L^2(D)}\|\theta\|_{L^2(D)},
\end{aligned}
\label{quadbound}
\end{equation}
where $c(\dd):=\frac{1}{\sqrt{(1+\cos \dd)/2}}$.
\end{theorem}
\begin{proof}
We start by showing that the conditions (a), (b) and (c) hold.
For (a), we note that  $g(\cdot;\eta,\theta)$ in analytic on $B$ if and only if the operator mapping 
$z \mapsto (e^{z}I-\Delta)^{-1}$ is analytic on $B$.
To see the latter, we fix $z_0 \in B$ and set $p_0 := e^{z_0}$. 
Clearly, 
$p_0 I-\Delta$ is invertible from $L^2(\RRD)$ to $L^2(\RRD)$. Let $M_0:= \| (p_0 I -\Delta)^{-1} \|_{L^2(\mathbb R^d) \to L^2(\mathbb R^d)}$.
For $p \in \mathbb C$, we write
$$
pI-\Delta = (p-p_0)I + (p_0I -\Delta) = (p_0I-\Delta) \left( (p-p_0)(p_0I-\Delta)^{-1} + I \right),
$$
so that the Neumann series representation
$$
(pI-\Delta)^{-1} = \left(\sum_{j=0}^\infty (-1)^j (p-p_0)^j (p_0I-\Delta)^{-j}\right)(p_0I-\Delta)^{-1}
$$
is uniformly convergent provided $\| (p-p_0) (p_0I-\Delta)^{-1}\|_{L^2(\mathbb R^d) \to L^2(\mathbb R^d)} < 1$ or
$$
| p - p_0 | < 1/M_0.
$$
Hence $(pI-\Delta)^{-1}$ is analytic in an open neighborhood of $p_0=e^{z_0}$ for all $p_0 \in B$ and (a) follows.

To prove (b) and (c), we first bound $g(z;\eta,\theta)$ for $z$ in the band $B$. 
Assume $\eta\in \widetilde{H}^\beta(D)$ and $\theta \in\tHs(D)$ with $\beta>s$.
For $z\in B$, we use the Fourier transform and estimate $|g|$ as follows
$$
\begin{aligned}
|g(z;\eta,\theta)|&= \left| e^{sz}\int_{\RRD} \frac{|\zeta|^2}{e^z+|\zeta|^2}  \mathcal F(\widetilde \eta) \overline{ \mathcal F(\widetilde \theta)} \, d\zeta \right|\\
&\le c(\dd) e^{s\Re z} \int_{\mathbb{R}^d}\frac{|\zeta|^2}{e^{\Re z}+|\zeta|^2}|\mathcal F(\widetilde \eta)| | \mathcal F(\widetilde \theta)| \, d\zeta ,\\
\end{aligned}
$$
where $c(\dd)=\frac{1}{\sqrt{(1+\cos \dd)/2}}$ and upon noting that
$$
	|e^z+|\zeta|^2|\ge c(\Im z)^{-1}(e^{\Re z}+|\zeta|^2)\ge c(\dd)^{-1}(e^{\Re z}+|\zeta|^2).
$$ 
If $\Re z<0$, we deduce that
\begin{equation}
|g(z;\eta,\theta)|\le c(\dd)e^{s\Re z}\|\eta \|_{L^2(D)}\|\theta\|_{L^2(D)}.
\label{yl0}
\end{equation}
Instead, when $\Re z\geq 0$, we write
$$
|g(z;\eta,\theta)|\le c(\dd)e^{(s-\delta)\Re z/2} \int_{\mathbb{R}^d} \frac{(|\zeta|^{2})^{1-(\delta+s)/2}(e^{\Re z})^{(\delta+s)/2}}{e^{\Re z}+|\zeta|^2}|\zeta|^{\delta+s}|\mathcal F(\widetilde \eta)| | \mathcal F(\widetilde \theta)| \, d\zeta.
$$
Whence,  Young's inequality guarantees that 
\begin{equation}
|g(z;\eta,\theta)|\le c(\dd)e^{(s-\delta)\Re z/2}\|\eta\|_{\widetilde{H}^\delta(D)} \|\theta\|_{\widetilde{H}^s(D)}.
\label{yg0}
\end{equation}

Gathering the above two estimates \eqref{yl0} and \eqref{yg0} gives
\begin{equation}
\int_{-\dd}^{\dd} |g(y+iw;\eta,\theta)|\, dw\leq 2\dd c(\dd)
\left\{ \begin{aligned}
\|\eta \|_{L^2(D)}\|\theta\|_{L^2(D)}, & \qquad y<0,\\
\|\eta\|_{\widetilde{H}^\delta(D)} \|\theta\|_{\widetilde{H}^s(D)}, &\qquad y\ge 0,
\end{aligned} \right.
\label{b-bound}
\end{equation}
and $N(B)$ in \eqref{inf} satisfies
\begin{equation}
N(B)\le c(\dd)(\frac{4}{\delta-s}\|\eta\|_{\widetilde{H}^\delta(D)} \|\theta\|_{\widetilde{H}^s(D)}
+\frac{2}{s}\|\eta \|_{L^2(D)}\|\theta\|_{L^2(D)}).\label{c-bound}
\end{equation}
Estimates \eqref{b-bound} and \eqref{c-bound} prove (b) and (c) respectively.

Having established (a), (b), and (c), we can use the sinc quadrature estimate \eqref{inf}.
In addition, from \eqref{yl0} and \eqref{yg0} we also deduce that
\begin{equation}\label{tail}
\begin{aligned}
k\sum_{j\le -\Nminus-1}^{-\infty} |g(kj;\eta,\theta)| &\le
\frac{c(\dd)}{s}e^{-s \Nminus k}\|\eta \|_{L^2(D)}\|\theta\|_{L^2(D)} \qquad\text{and }
	\\ k\sum_{j\ge\Nplus+1}^\infty |g(kj;\eta,\theta)| &\le
	\frac{2c(\dd)}{\delta-s}e^{(s-\delta)\Nplus k/2}\|\eta\|_{\widetilde{H}^\delta(D)} \|\theta\|_{\widetilde{H}^s(D)} .
	\end{aligned}
	\end{equation}
Combining \eqref{inf} with \eqref{c-bound} and \eqref{tail} shows \eqref{quadbound} and completes the proof.
\end{proof}

\begin{remark}[Choice of $N^-$ and $N^+$]\label{r:balanced} Balancing the three exponentials in \eqref{quadbound} leads to the following choice
$$2\pi\dd/k\approx (\delta-s)\Nplus k/2 \approx s\Nminus k .$$
Hence, for given the quadrature spacing $k>0$, we set
\begin{equation}\label{e:choiceN}
\Nplus := \bigg\lceil\frac{4\pi\dd}{k^2(\delta-s)}\bigg\rceil\qquad\hbox{and}\qquad \Nminus:=\bigg\lceil\frac{2\pi\dd}{sk^2}\bigg\rceil.
\end{equation}
With this choice,   \eqref{quadbound} becomes
\begin{equation}\label{e:quad_cons}
|a(\eta,\theta)-a^k(\eta,\theta)| \leq \gamma(k) \| \eta \|_{\widetilde{H}^\delta(D)} \| \theta \|_{\widetilde H^s(D)}
\end{equation}
where
\begin{equation}\label{e:gammak}
\gamma(k) := C\left(\frac{1}{\delta-s},\frac{1}{s},\dd\right)e^{-2\pi\dd/k}.
\end{equation}
\end{remark}
   

\section{Truncated Domain Approximations}\label{s:truncated}
To develop further approximation to problem \eqref{mainWP} based on the sinc quadrature approximation \eqref{ik}, we  replace \eqref{vt}
with problems on bounded domains.

\subsection{Approximation on Bounded Domains} \label{s:truncated_setting}

Let $\Omega$ be a convex bounded domain containing $D$ and the origin. 
Without loss of generality, we assume that the diameter of $\Omega$ is 1. 
This auxiliary domain is used to generate suitable truncation domains to approximate the solution of \eqref{vt}.
We introduce a domain parameter $M>0$ and define the dilated domains
\begin{equation}\label{e:omega_Mt}
\Omega^M(t):= \left\lbrace
\begin{array}{ll}
\left\lbrace y = (1+t(1+M))x \ : \ x \in \Omega \right\rbrace, &\qquad t\geq 1,\\
\left\lbrace y = (2+M)x \ : \ x \in \Omega \right\rbrace, &\qquad t < 1.
\end{array}\right. 
\end{equation}

The approximation of $a^k(\cdot,\cdot)$ in \eqref{ik} reads
\begin{equation}
a^{k,M}(\eta,\theta):= \frac{c_s k}{2}\sum_{j=-\Nminus}^\Nplus e^{\beta y_j}(w^M(\widetilde\eta, t_j),\theta)_{ {D}},
\label{isM}
\end{equation}
with $t_j :=t(y_j)=e^{-y_j/2}$, according to \eqref{e:tj}, and
\begin{equation}\label{wMt}
w^{M}(t):=w^M(\widetilde \eta,t)=\widetilde\eta|_{\Omega^M(t)}+v^M(\widetilde\eta,t),
\end{equation}
where $v^M(t):=v^M(\widetilde\eta,t)$ solves
\begin{equation}
(v^M(t),\phi)_{\Omega^M(t)}+t^2 d_{\Omega^M(t)}(v^M(t),\phi)=-(\eta,\phi)_D,\Forall \phi \in
H^1_0(\Omega^M(t));
\label{vtM}
\end{equation}
compare with \eqref{vt}.
The domains $\Omega^M(t_j)$ are constructed for the truncation error to be exponentially decreasing as a function of $M$. This is the subject of next section. 

\subsection{Consistency}\label{s:consistency_truncated}

The main result of this section provides an  estimate for $a^k-a^{k,M}$.
It relies on decay properties of  $v(\widetilde \eta,t)$ satisfying \eqref{vt}.
In fact, Lemma 2.1 of \cite{auscher2002solution}  
guarantees the existence of universal constants $c$ and
$C$ such that 
\begin{equation}\label{ineq:kato_esti}
	t\|\nabla v(\widetilde \eta,t)\|_{L^2(B^M(t))}+\|v(\widetilde \eta,t)\|_{L^2(B^M(t))}\leq Ce^{-\max(1,t)cM/t }\|\eta\|_{L^2(D)},
\end{equation}
provided  $\eta \in L^2(D)$ and $v(t):= v(\widetilde \eta,t)$ is given in \eqref{vt}.
Here 
$$
B^M(t):= \{x\in \Omega^M(t) \ :\
\hbox{dist}(x,\partial\Omega^M(t))<t\}
$$
so that the minimal distance between points in $D\subset \Omega$ and $B^M(t)$ is greater than $M\max(1,t)$.
An illustration of the different domains is provided in Figure~\ref{f:domains}.

\begin{figure}[ht!]
\begin{center}
\scalebox{1.0}{
\includegraphics[width=0.5\textwidth]{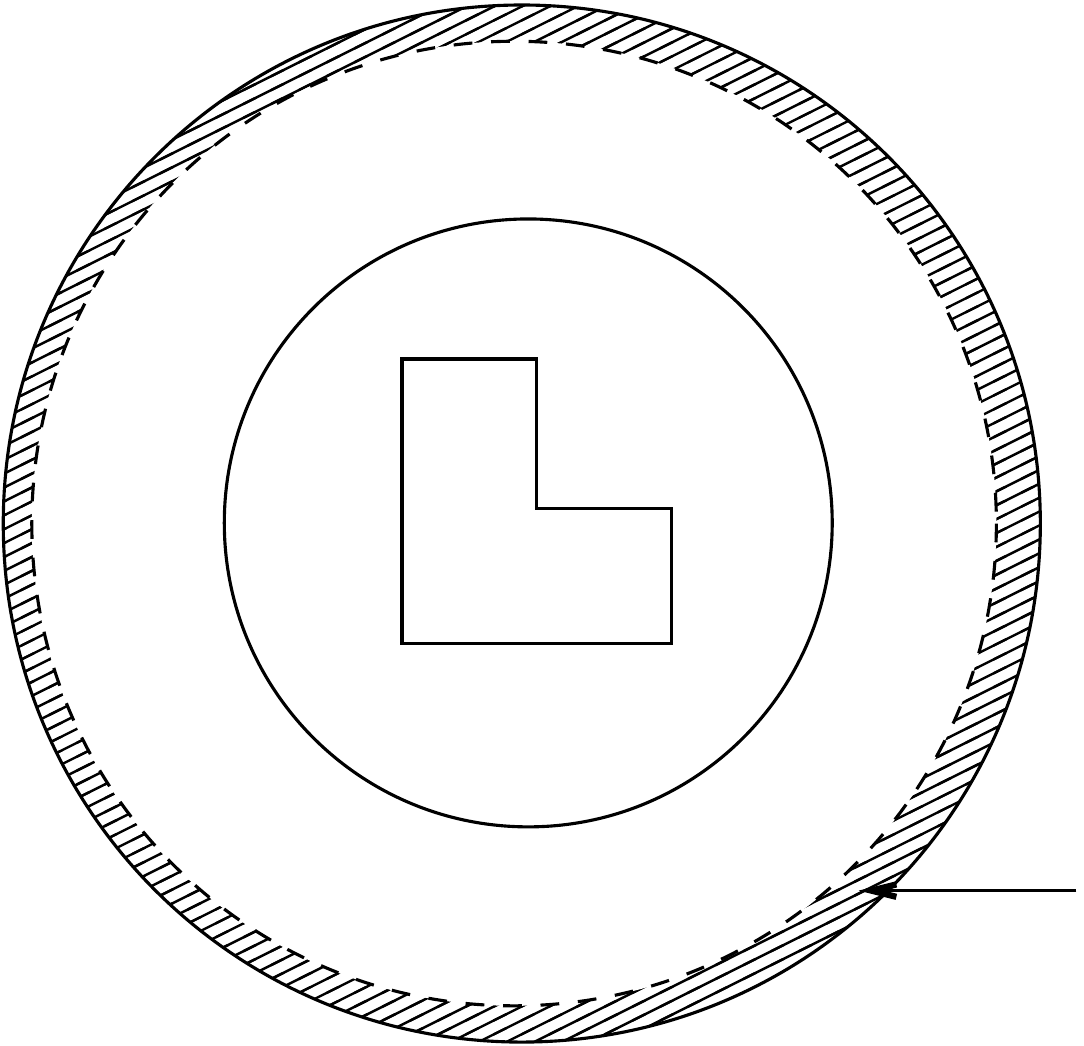}
\begin{picture}(0,0)
\put(0,25){$B^M(t)$}
\put(-100,75){$D$}
\put(-75,110){$\Omega$}
\put(-20,150){$\Omega^M(t)$}
\end{picture}
}
\caption{Illustration of the different domains in $\mathbb R^2$. The domain of interest $D$ is a L-shaped domain, $\Omega \subset \Omega^M(t)$ are interior of discs, and $B^M(t)$ is the filled portion of $\Omega^M(t)$.}\label{f:domains}
\end{center}
\end{figure}
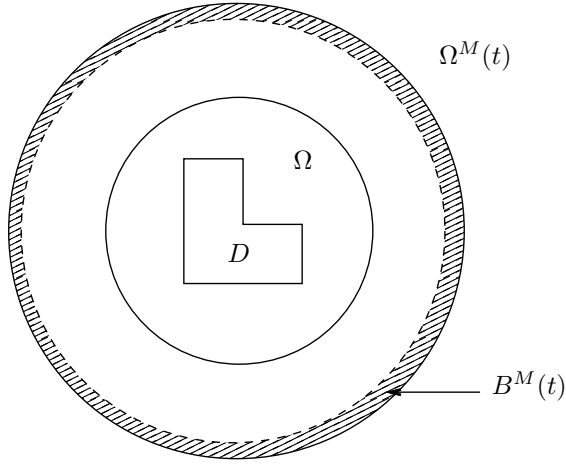

\begin{lemma}[Truncation error] \label{l:deday}
Let $\eta \in L^2(D)$, $e(t):=v(\tilde \eta, t)-v^M(\tilde \eta,t)$ and $c$ be the constant appearing in \eqref{ineq:kato_esti}.
There is a positive  constant $C$ not depending
  on $M$ and t satisfying 
  \begin{equation}\label{domainbound2}
  \|e(t)\|_{L^2(\Omega^M(t))}\leq Ce^{-\max(1,t)cM/t }\|\eta\|_{L^2(D)} .
  \end{equation}

\end{lemma}
\begin{proof} 
In this proof $C$ denotes a generic constant only depending on $\Omega$.
Note that $e(t)$ satisfies the relations
\begin{equation}
\begin{aligned}
(e(t),\phi)+t^2d_{\Omega^M(t)}(e(t),\phi)&=0, \quad \forall \phi\in H^1_0(\Omega^M(t)),\\
e(t)&=v(t),\quad\hbox{on }\partial\Omega^M(t).
\end{aligned}
\label{etform}
\end{equation}

Let $\chi(t) \geq 0$ be a bounded cut off function
satisfying $\chi(t)=1$ on $\partial \Omega^M(t)$ and $\chi(t)=0$ on
$\Omega^M(t)\setminus B^M(t)$. 
Without loss of generality,
we may assume that 
$\|\nabla \chi(t)\|_{L^\infty(\RRD)} \le C/t$.  
This implies
\begin{equation*}
\begin{split}  
\|\chi(t) v(t)\|_{L^2(B^M(t))}&+t\|\nabla(\chi(t)
v(t))\|_{L^2(B^M(t))} \\
&\leq C(\|v(t)\|_{L^2(B^M(t))}+t\|\nabla v(t)\|_{L^2(B^M(t))})\\
&\leq Ce^{-\max(1,t)cM/t }\|\eta\|_{L^2(D)}.
\end{split}
\end{equation*}
Here we use the decay estimate~\eqref{ineq:kato_esti} for last inequality above.
Now, setting $e(t):=\chi(t) v(t)+\zeta(t)$, we find that
$\zeta(t)\in H^1_0(\Omega^M(t))$
satisfies 
$$
(\zeta(t),\phi)_{\omt}+t^2d_{\Omega^M(t)}(\zeta(t),\phi)=-(\chi(t) v(t),\phi)_{\omt}-t^2d_{\Omega^M(t)}(\chi(t) v(t),\phi)
$$
for all $\phi\in H^1_0(\Omega^M(t))$. Taking $\phi=\zeta(t)$, we deduce that
$$
\begin{aligned} \|\zeta(t)\|^2_{L^2(\Omega^M(t))}+t^2\|\nabla\zeta(t)\|^2_{L^2(\Omega^M(t))}&\leq\|\chi(t) v(t)\|^2_{L^2(B^M(t))}+t^2\|\nabla(\chi(t) v(t))\|^2_{L^2(B^M(t))}\\
&
\leq Ce^{-2\max(1,t)cM/t }\|\eta\|^2_{L^2(D)}.\end{aligned}$$
Thus, combining the estimates for $\zeta(t)$ and $\chi(t)v(t)$ completes the proof.
\end{proof}
Lemma~\ref{l:deday} above is instrumental to derive exponentially decaying consistency error as $M\to \infty$. Indeed,
we have the following theorem.

\begin{theorem}[Truncation error]\label{t:truncate}
Let $c$ be the constant appearing in \eqref{ineq:kato_esti} and assume $M>2(s+1)/c$. Then,
there is a positive constant $C$ not depending
  on $M$ nor $k$ satisfying 
\begin{equation}
|a^k(\eta,\theta)-a^{k,M}(\eta,\theta)|\le
Ce^{-cM}\|\eta\|_{L^2(D)}\|\theta\|_{L^2(D)},\Forall \eta,\theta\in
L^2(D).\label{truncateb}
\end{equation}
\end{theorem}
\begin{proof}
In this proof $C$ denotes a generic constant only depending on $\Omega$. Let $\eta,\theta$ be in 
$L^2(D)$. It suffices to bound
$$
\begin{aligned}
E&:= \left| \frac{c_s k}{2}\sum_{j=-\Nminus}^\Nplus e^{s y_j}(w(t_j)-w^M(t_j),\theta)_D\right|\\
&\le C\left(k\sum_{j=-\Nminus}^{-1} e^{s y_j}|(v(t_j)-v^M(t_j),\theta)_D|+k\sum_{j=0}^{\Nplus} e^{s y_j}|(v(t_j)-v^M(t_j),\theta)_D|\right)\\
&=: E_1+E_2 
\end{aligned}
$$
with $v(t)=v(\widetilde\eta,t)$ defined by \eqref{vt} and $v^M(t)=v^M(\widetilde\eta,t)$ defined by \eqref{vtM}.
We estimate $E_1$ and $E_2$ separately, starting with $E_1$.

From the definition $t_j=e^{-y_j/2}$, we deduce that when $j<0$,  $t_j>1$ so that \eqref{domainbound2} gives
$$
\begin{aligned}
E_1&\leq Cke^{-cM}\sum_{j=-\Nminus}^{-1}e^{s y_j}\|\eta\|_{L^2(D)}\|\theta\|_{L^2(D)} \\
&\leq Ce^{-cM}\frac{ke^{-s k}}{1-e^{-s k}}\|\eta\|_{L^2(D)}\|\theta\|_{L^2(D)} \leq Ce^{-cM}\|\eta\|_{L^2(D)}\|\theta\|_{L^2(D)}  .
\end{aligned}
$$
Similarly, for $j\geq 0$, i.e. $t_j<1$, using \eqref{domainbound2} again,  we have
$$
\begin{aligned}
E_2&\leq Ck\sum_{j=0}^{\Nplus}e^{s y_j}e^{-{c M}/t_j}\|\eta\|_{L^2(D)}\|\theta\|_{L^2(D)} \\
&\leq Ck\sum_{j=0}^{\Nplus}e^{s y_j}e^{-cM(1+y_j/2)}\|\eta\|_{L^2(D)}\|\theta\|_{L^2(D)}\\
&= Cke^{-cM}\sum_{j=0}^{\Nplus}e^{(s-cM/2) y_j}\|\eta\|_{L^2(D)}\|\theta\|_{L^2(D)}\\
&\leq Ce^{-cM}\frac{k}{1-\exp(k(s-cM/2))}\|\eta\|_{L^2(D)}\|\theta\|_{L^2(D)}\\
& \leq \frac{Ce^{-cM}}{cM/2-s}\|\eta\|_{L^2(D)}\|\theta\|_{L^2(D)}\le C e^{-cM}\|\eta\|_{L^2(D)}\|\theta\|_{L^2(D)},
\end{aligned}
$$
where we have also used the property $cM/2-s>1$ guaranteed by the assumption $M> 2(s+1)/c$.
\end{proof}

\subsection{Uniform Norm Equivalence on Convex Domains}

Since the domains $\Omega^M(t)$ are convex, we know that the norms in $\dH r(\Omega^M(t))$ are equivalent to those in $H^r(\Omega) \cap H^1_0(\Omega^M(t))$ for $r \in [1,2]$, see e.g. \cite{BP13}.
However, as we mentioned in Remark~\ref{r:beta-equiv}, the equivalence 
constants depend a-priori on $\Omega^M(t)$ and
therefore on $M$ and $t$. We show in this section that they can be bounded
uniformly independently of both parameters.

To simplify the notation introduced in Section~\ref{s:dotted_space}. 
We shall denote  $T_\OMt$ by  $T_t$,
$L_{\omt}$ by $L_t$  and $\dH s(\omt)$ by $\dH s$.
We recall that $\Omega^M(t)$ is a dilatation of the convex and  bounded domain $\Omega$ containing the origin, see \eqref{e:omega_Mt}. We then have the following lemma.

\def\vhmt{{{\mathbb V}_h^M(t)}}
\begin{lemma}[Ellipitic Regularity on Convex Domains] \label{h2omt} Let  $f\in L^2(\omt)$. Then  $\theta:=T_t f$ is in
  $H^2(\Omega^M(t))\cap H^1_0(\Omega^M(t))$
and satisfies 
\begin{equation}
\|\theta \|_{H^2(\omt)}\le C\|f\|_{L^2(\omt)},
\label{uh2}
\end{equation}
where $C$ is a constant independent of $t$ and $M$.
\end{lemma}

\begin{proof} It is well known that the convexity of $\Omega$ and hence
  that of $\omt$ implies that the unique solution $\theta$ of
  \eqref{varscale} with $\omega$ replaced by $\omt$  
is in $H^2(\Omega^M(t))\cap
  H^1_0(\Omega^M(t))$.  
Therefore,  the crucial point is to show that the constant in \eqref{uh2} does not depend on $M$ or $t$. 
To see this, the $H^2$ elliptic regularity on convex domains implies that  for $\hat \theta \in H^1_0(\Omega)$ with $\Delta \hat \theta \in L^2(\Omega)$ then $\hat \theta \in H^2(\Omega)$ and there is a constant $C$ only depending on $\Omega$ such that
\begin{equation} |\hat \theta|_{H^2(\Omega)} \le C \|\Delta \hat \theta\|_{L^2(\Omega)}.
\label{h2sreg}
\end{equation} 
Here $|\cdot|_{H^2(\Omega)}$ denotes the $H^2(\Omega)$ seminorm.
Let $\gamma$ be such that $\omt=\{\gamma x,\
  x\in \Omega\} $ (see \eqref{e:omega_Mt}) and $\hat \theta(\hat x) = \theta(\gamma \hat x)$ for $\hat x\in
  \Omega$. 
Once scaled back to $\Omega^M(t)$, estimate \eqref{h2sreg} gives 
\begin{equation}
| \theta |_{H^2(\omt)} \le C\|\Delta \theta \|_{L^2(\omt)} =
C\|f-\theta\|_{L^2(\omt)}.\label{h2sreg1}
\end{equation}
Now \eqref{varscale} immediately implies that 
$\|\theta \|_{H^1(\omt))}\le \|f\|_{L^2(\omt)}$ and \eqref{uh2} follows by
the triangle inequality and obvious manipulations.
\end{proof} 

\begin{remark}[Intermediate Spaces]\label{r:h2omt} Lemma~\ref{h2omt} implies that
  $D(L_t)=\dH 2=H^2(\omt)\cap H^1_0(\omt)$ with norm equivalence constants independent of $M$ and $t$.   As 
$D(L_t^{1/2}) = \dH 1 =H^1_0(\omt)$, for $s \in [1,2]$
$$\dH s = (H^1_0(\omt),H^2(\omt)\cap H^1_0(\omt))_{s-1,2} {=H^s(\omt)\cap H^1_0(\omt)}$$
with with norm equivalence constants independent of $M$ and $t$.
\end{remark}

\begin{lemma}[Norm Equivalence]  \label{hdbound}  For $\beta\in [1,3/2)$, let $\theta$ be in $\dH \beta$
and $\widetilde \theta$ denote its extension by zero outside of $\OMt$.   Then
$\widetilde \theta$ is in $H^\beta(\RRD)$ and 
$$\|\theta \|_{\dH \beta } \le C \|\widetilde
\theta \|_{H^\beta(\RRD)}$$
with $C$ not depending on $t$ or $M$.
\end{lemma} 

\begin{proof}  
Given $\theta \in H^1(\omt)$, we denote $R\theta $ to be the
  elliptic projection of $\theta$ into $H^1_0(\omt)$, i.e., $R\theta \in
  H^1_0(\omt)$ is the solution of 
$$\begin{aligned} (R\theta,\phi)_{\omt} &+d_\omt(R\theta,\phi)\\
&=(\theta,\phi)_{\omt}
+d_\omt(\theta,\phi)
,\Forall \phi \in
H^1_0(\omt).\end{aligned}
$$
It immediately follows that 
$$\|R\theta\|_{\dH1} = \|R\theta \|_{H^1(\omt)} \le  \|\theta\|_{H^1(\omt)}.
$$
Also, if $\theta \in H^2(\omt)$, Lemma~\ref{h2omt} (see also Remark~\ref{r:h2omt}) implies
$$\|R\theta\|_{\dH2} \le C\|R\theta\|_{H^2(\omt)} \le C \|\theta\|_{H^2(\omt)}$$
with $C$ not depending on $t$ or $M$. Hence, it follows by interpolation that
\begin{equation}
\|R \theta \|_{\dH \beta} \le C_\beta \|\theta\|_{(H^1(\omt),H^2(\omt))_{\beta-1,2}}.
\label{feq}
\end{equation}

Now when $\theta \in \dH \beta \subset H^1(\Omega^M(t))$, $R\theta = \theta$ so that in view of \eqref{feq}, it remains to show that
$$
\|\theta\|_{(H^1(\omt),H^2(\omt))_{\beta-1,2}} \leq C \| \tilde \theta \|_{H^\beta(\RRD)},
$$
for a constant $C$ independent of $M$ and $t$.
To see this, note that $\widetilde \theta$ is in
$H^1(\RRD)$ and the extension of $\nabla \theta$ by zero is in $H^{\beta-1}(\RRD)$ for $\beta < 3/2$.  
We refer to Theorem 1.4.4.4 of \cite{grisvard} for a proof when $d=1$ and the techniques used in Lemma 4.33 
of \cite{DD12} for the extension to the higher dimensional spaces.  
This implies that $\widetilde \theta$ belongs to $H^\beta(\RRD)$.  
Moreover, the restriction operator is
simultaneously bounded from $H^j(\RRD)$ to $H^j(\omt)$ for $j=1,2$.
Hence, by interpolation again, we have that
$$\|\theta \|_{(H^1(\omt),H^2(\omt))_{\beta-1,2}} \le \|\widetilde \theta\|_{H^\beta(\RRD)}.$$
This completes the proof of the lemma.
\end{proof}


\section{Finite Element Approximation}
\label{s:FEM}
In this section, we turn our attention to the finite element approximation of each subproblems \eqref{vtM} in $a^{k,M}(\cdot,\cdot)$. 
Throughout this section, we omit when no confusion is possible the subscript $j$ in $t_j$, i.e. we consider a generic $t$ keeping in mind that the subsequent statements only hold for $t=t_j$ with $j=-N^-,...,N^+$.
We also make the additional unrestrictive assumption that $\Omega$ used to define $\Omega^M(t)$ (see \eqref{e:omega_Mt}) is polygonal.
In turn, so are all the dilated domains $\Omega^M(t)$. 

\subsection{Finite Element Approximation of $a^{k,M}(\cdot,\cdot)$}
For any polygonal domain $\omega$, let  $\{\mathcal{T}_h(\omega) \}_{h>0}$ be a sequence of conforming subdivisions
 of $\omega$ made of simplices of maximal size diameter $h\le1$. We
 use the notation $\mathcal T^M_h(t):=\mathcal{T}_h(\Omega^M(t))$
 for $t=t_j$, $j=-N^-,...,N^+$, given by \eqref{e:tj}.   We assume that the subdivisions
 on $D$
are  shape-regular and quasi-uniform.  This
means that there exist universal constants  $\sigma,\rho >0$ such that
  \begin{equation}\label{a:shape-regular}
\sup_{h>0}\max_{T\in\mathcal T_h(D)}\left(\frac {\text{diam}(T)}{r(T)}\right)
 \leq \sigma,
\end{equation}
  \begin{equation}\label{a:quasi-uniform}
\sup_{h>0} \left(\frac{\max_{T \in \mathcal T_h(D)} \text{diam}(T)} {\min_{T \in \mathcal T_h(D)} \text{diam}(T)}\right) \leq \rho,
\end{equation}
where $\text{diam}(T)$ stands for the diameter of $T$ and $r(T)$ for the
radius of the largest ball contained in $T$.   We also assume that these
conditions hold as well for  $\mathcal T^M_h(t_j)$ with constants
$\sigma,\rho$ not depending on $j$.
We finally require that all the subdivisions match on $D$, i.e.  
\begin{equation}
\mathcal
T_h(D)\subset \mathcal T^M_h(t_j)
\label{a:mesh}
\end{equation}
for each $j$.
We discuss in Section~\ref{s:numerical} how to generate subdivisions meeting these requirements.

Define $\mathbb{V}_h(\omega )\subset H^1_0(\omega)$ to be the space of 
continuous piecewise linear finite element functions associated with
$\mathcal{T}_h(\omega)$ with $\omega =D$ or $\Omega^M(t)$.
Also, we use the short notation $\mathbb V_h^M(t):=\mathbb V_h(\Omega^M(t))$.

We are now in position to define the fully discrete/implementable problem. 
For $\eta_h$ and $\theta_h$ in $\vh {(D)}$, the finite element approximation of $a^{k,M}(\cdot,\cdot)$ given by \eqref{isM} is 
\begin{equation}
a_h^{k,M}(\eta_h,\theta_h):=\frac{c_s k}{2}\sum_{j=-\Nminus}^\Nplus e^{s y_j}(w_h^M(\widetilde\eta_h,t_j),\theta_h)_{ {D}}
\label{isMh}
\end{equation}
with 
\begin{equation}\label{e:wMh}
w^M_h(\widetilde \eta_h,t):=\widetilde \eta_h|_{\Omega^M(t)} +v^M_h(t)
\end{equation} 
and where $v^M_h(t) \in \mathbb V_h^M(t)$ solves 
\begin{equation}\label{e:vhM}
	(v^M_h(t),\phi_h)_{\Omega^M(t)}+t^2d_{\Omega^M(t)}(v^M_h(t),\phi_h)=-(\widetilde \eta_h,\phi_h)_{\Omega^M(t)},\quad\forall \phi_h\in \mathbb V^M_h(t).
\end{equation}

\begin{remark}[Assumption \eqref{a:mesh}]
Two critical properties follow from \eqref{a:mesh}.
On the one hand, our analysis below relies on the fact that the extension by zero $\widetilde
v_h$ of $v_h \in \mathbb V_h {(D)}$ belongs to all $\mathbb V_h^M(t)$. 
This property greatly simplifies the computation of
$(w_h^M(\widetilde\eta_h,t_j),\theta_h)_{ {D}}$ in \eqref{isMh}.
\end{remark}

The finite element approximation 
of the problem \eqref{mainP} is to find $u_h\in \mathbb{V}_h {(D)}$ so that
\begin{equation}
a^{k,M}_h(u_h,\theta_h)=(f,\theta_h)_{ {D}}\qquad\hbox{for all }
\theta_h\in\mathbb{V}_h {(D)} .
\label{finaldiscrete}
\end{equation}

 Analogous 
to Lemma~\ref{l:K}, we have the following representation using K-functional.  The proof of the lemma is similar to that of
Lemma~\ref{l:K} and is omitted.

\begin{lemma}[K-functional formulation on the discrete space]\label{l:K_discrete}
For $\eta_h \in \vh {(D)}$, there holds
$$
(w_h^M(\widetilde \eta_h,t), \eta_h)_{D} = (w_h^M(\widetilde \eta_h,t), \widetilde \eta_h)_{\Omega^M(t)} = \ K_h(\widetilde \eta_h,t),
$$
where
$$
K_h(\widetilde \eta_h,t):= \min_{\varphi_h \in \mathbb V_h^M(t)} \left( \| \widetilde \eta_h - \varphi_h \|_{L^2(\Omega^M(t))}^2 + t^2 d_{\omt}(\varphi_h,\varphi_h) \right).
$$
\end{lemma}

We emphasize that for $v_h \in \mathbb V_h^M(t)$, its extension by zero $\tilde \eta_h$ belongs to $H^1(\RRD)$ and therefore
\begin{equation}\label{e:monotonK}
K_h(\tilde v_h,t)  \geq K(\tilde v_h,t).
\end{equation}
This property is critical in the proof of next theorem, which ensures the $\vh(D)$-ellipticity of the discrete bilinear for  $a_h^{k,M}$.
Before describing this next result, we recall that according to \eqref{e:quad_cons}
$$
|a(\eta_h,\theta_h)-a^k(\eta_h,\theta_h)|\le\gamma(k) \|\eta_h\|_{\widetilde H^\delta(D)}\|\theta_h\|_{\tHs(D)}
$$
with 
  $\delta$ between $s$ and $\min(2-s,3/2)$
(since $\vh(D)\subset \widetilde H^{3/2-\epsilon}(D)$ for any $\epsilon>0$) and $\gamma(k) \sim Ce^{-2\pi\dd/k}$.
Also, we note that from the quasi-uniform  \eqref{a:shape-regular} and shape-regular \eqref{a:quasi-uniform} assumptions, there exists a constant $c_I$ only depending on $\sigma$ and $\rho$ such that for $r^- \leq r^+ < 3/2$, there holds
\begin{equation}\label{e:inverse}
\| v_h \|_{\widetilde H^{r^+}(D)} \leq c_I h^{r^{-}-r^+}\| v_h \|_{\widetilde H^{r^-}(D)}, \qquad \forall v_h \in \mathbb V_h^M(t).
\end{equation}
\begin{theorem}[$\vh {(D)}$-ellipticity]\label{l:ellipticity}
Let $\delta$ in Theorem~\ref{t:quad_error} between $s$ and
  $\min(2-s,3/2)$, $k$ be the quadrature spacing and
 $c_I$ be the inverse constant in \eqref{e:inverse}. We assume that the quadrature parameters $N^-$ and $N^+$ are chosen according to \eqref{e:choiceN}.  Let $\gamma(k)$ be given by \eqref{e:gammak} and assume that $k$ is chosen sufficiently small so that
$$
c_I \gamma(k) h^{s-\delta} < 1.
$$
Then, there is a constant $c$ independent of $h,k$ and $M$ such that
$$
a_h^{k,M}(\eta_h,\eta_h) \geq c\| \eta_h \|_{\tHs(D)}^2, \Forall \eta_h \in \vh(D).
$$
\end{theorem}

\begin{proof}
Let $\eta_h \in \mathbb V_h^M(t)$ so that $\widetilde \eta_h\in H^1(\RRD)$.
We use the equivalence relations provided by Lemmas~\ref{l:K} and~\ref{l:K_discrete} together with the monotonicity property \eqref{e:monotonK} to write
$$
a_h^{k,M}(\eta_h,\eta_h) = \frac{c_s k}{2}\sum_{j=-\Nminus}^\Nplus e^{s y_j}K_h(\widetilde \eta_h,t_j) \geq 
\frac{c_s k}{2}\sum_{j=-\Nminus}^\Nplus e^{s y_j}K(\widetilde \eta_h,t_j) = a^k(\eta_h,\eta_h).
$$
The quadrature consistency bound \eqref{e:quad_cons} supplemented by an inverse inequality~\eqref{e:inverse} yields
$$
a_h^{k,M}(\eta_h,\eta_h)  \geq a(\eta_h,\eta_h)-\gamma(k) \|\eta_h\|_{\widetilde H^\delta(D)}\|\eta_h\|_{\tHs(D)}
\geq a(\eta_h,\eta_h)- c_I \gamma {(k)} h^{s-\delta}\|\eta_h\|_{\tHs(D)}^2.
$$
The desired result follows from assumption  $c_I \gamma(k)
h^{s-\delta}<1
$ and the coercivity of $a(\cdot,\cdot)$, see \eqref{e:coercivity_a}. 
\end{proof}

\subsection{Approximations on $\Omega^M(t)$}

The fully discrete scheme \eqref{finaldiscrete} requires approximations by the finite element methods on domains $\Omega^M(t)$. 
Standard finite element argumentations would lead to estimates with constants depending on $\Omega^M(t)$ and therefore $M$ and $t$.
In this section, we exhibit results where this is not the case due to the particular definition \eqref{e:omega_Mt} of $\Omega^M(t)$.

We can use interpolation to develop approximation results for functions
in the intermediate spaces with constants independent of $M$ and $t$.  The
Scott-Zhang interpolation construction  {\cite{SZ90}} gives rise to an approximation
operator $\pi^{sz}_h:H^1_0(\omt)\rightarrow \vh^M(t)$ satisfying 
$$ \|\eta- \pi^{sz}_h \eta \|_{H^1(\omt)} \le C \|
\eta\|_{H^1(\omt)},
$$
for all $\eta\in H^1_0(\omt)=\dH 1$ 
and 
$$\| \eta-\pi^{sz}_h \eta \|_{H^1(\omt)} \le C h \|
\eta\|_{H^2(\omt)},$$
for all $\eta\in H^2(\omt)\cap H^1_0(\omt)=\dH 2$.
The Scott-Zhang argument is local so the constants appearing above depend
on the shape regularity of the triangulations but not on $t$ or $M$.
Interpolating the above inequalities shows that for all $r\in [0,1]$
\begin{equation}
\inf_{\chi\in \vhmt} \|\eta-\chi\|_{H^1(\omt)} \le C h^r \|\eta\|_{\dH
  {1+r}}, \Forall \eta\in \dH {1+r}
\label{app}
\end{equation}
with $C$ not depending on $t$ or $M$.

Let $T_{t,h}$ denote the finite element approximation to $T_t$ given by \eqref{e:Tomega}, i.e., 
for $F\in \dH {-1}$, $T_{t,h}F:=w_h$ with $w_h\in \vh^M(t)$ being the unique
solution of
$$
 (w_h,\phi)_\omt +d_\omt (w_h,\phi)  = \langle F,\phi\rangle, \Forall \phi\in \vh^M(t).
$$
The approximation result \eqref{app}  and standard finite element analysis
techniques implies that for any $r\in [0,1]$,
\begin{equation}
\|T_tF - T_{t,h}F\|_{L^2(\omt)} \le C h^{1+r} \|T_tF\|_{\dH {1+r}}\le  C h^{1+r} \|F\|_{\dH {-1+r}},
\label{Tdifapp}
\end{equation}
where the last inequality follows from interpolation since $\|T_t F\|_{H^1(\omt)}\le \|F\|_{H^{-1}(\omt)}$ and \eqref{uh2} hold.

For $f\in L^2(\Omega^M(t))$, we define the operator 
\begin{equation}\label{e:Sop}
\tzer_tf:=\eta \in
H^1_0(\omt)
\end{equation} 
satisfying, 
$$d_\omt(\eta,\phi)= (f,\phi)_\omt,\Forall \phi\in H^1_0(\omt)$$
and let $\tzer_{t,h} f\in \vh^M(t)$ denote its finite element
approximation; compare with $T_t$ and $T_{h,t}$.  Although the Poincar\'e constant depends on the diameter of $\omt$, we still have the following lemma.

\begin{lemma} \label{tzer} There is a constant $C$ independent of $h$, $t$, or $M$
  satisfying
$$\|\tzer_t f-\tzer_{t,h} f\|_{L^2(\omt)} \le Ch^2 \|f\|_{L^2(\omt)}.$$
\end{lemma}

\begin{proof}
For $f\in L^2(\omt)$, set $e_h:=(\tzer_t  -\tzer_{t,h}) f$.
 The elliptic regularity estimate \eqref{h2sreg1} on convex domain and Cea's Lemma imply
\begin{equation*}
\begin{split}
|e_h |_{H^1(\omt)} &= \inf_{\chi_h\in \vhmt}
|\tzer_t f -\chi_h |_{H^1(\omt)}
\le Ch |\tzer _t  f|_{H^2(\omt)} \\
&\le Ch \|\Delta \tzer _t  f\|_{L^2(\omt)} = Ch \| f\|_{L^2(\omt)},
\end{split}
\end{equation*}
where $C$ is a constant independent of $h$, $t$ and $M$.
Galerkin orthogonality and the above estimate give
$$\begin{aligned} \|e_h\|_{L^2(\omt)} ^2&=d_\omt(e_h,\tzer_t e_h)=
d_\omt(e_h,(\tzer_t-\tzer_{t,h})e_h )\\
&\le |e_h|_{H^1(\omt)} |(\tzer_t-\tzer_{t,h})e_h|_{H^1(\omt)} \\& \le Ch
|e_h|_{H^1(\omt)}  \|e_h\|_{L^2(\omt)} .
\end{aligned}
$$
Combining the above two inequalities and obvious manipulations completes
the proof of the lemma.
\end{proof}

We shall also need norm equivalency on discrete scales.  Let
$(\vhmt,\|\cdot\|_{L^2(\omt)})$ and $(\vhmt,\|\cdot\|_{H^1(\omt)})$ denote
$\vhmt$ normed with the norms in $L^2(\omt)$ and $H^1(\omt)$, respectively.
We define $\|\cdot\|_{\dH r_h(\Omega^M(t))}$, or simply $\|\cdot\|_{\dH r_h}$, to be the norm 
in the interpolation space
$$
\big(\vhmt,\|\cdot\|_{L^2(\omt)}),(\vhmt,\|\cdot\|_{H^1(\omt)})\big)_{r,2}.
$$
For $r\in[0,1]$, as the  natural injection is a bounded map (with bound 1) from
$\vhmt$ into $L^2(\omt)$ and $H^1_0(\omt)$, respectively, 
$\|v_h\|_{\dH r} \le \|v_h\|_{\dH r_h}$, for all $v_h\in \vhmt$.   
For the other direction, one needs a projector into $\vhmt$ which is simultaneously
bounded on $L^2(\omt)$ and $H^1_0(\omt)$.  In the case of a globally quasi uniform
mesh, it was shown by Bramble and Xu \cite{BX91} that the $L^2(\omt)$ projector
$\pi_h$ satisfies this property.   Their argument is local, utilizing the
inverse inequality \eqref{e:inverse} and hence leads to constants depending on those 
 appearing in \eqref{a:shape-regular} and \eqref{a:quasi-uniform} 
 but not $t$, $h$, or $M$.  Interpolating
these results gives, for
$r\in [0,1]$,
\begin{equation}
c\|v_h\|_{\dH r_h} \le \|v_h \|_{\dH r} \le \|v_h\|_{\dH r_h} , \Forall
v_h\in \vhmt,
\label{dsplus}\end{equation}
where $c$ is a constant independent of $h$, $M$ and $t$.
The spaces for negative $r$ are defined by duality and the stability of the $L^2(\Omega^M(t))$-projection $\pi_h$ yields
\begin{equation}
c\|v_h\|_{\dH{-r}} \le \|v_h\|_{\dH {-r}_h} \le
\|v_h\|_{\dH {-r}}.\Forall v_h\in \vhmt.
\label{dsneg}
\end{equation}

We finally note that a discrete version of Lemma~\ref{Tshift} holds.
Its proof is essentially the same and is omitted for brevity. 

\begin{lemma} \label{Tshifth} Let $a$ be in $[0,2]$ and $b$ be in $[0,1]$
  with $a+b\le 2$.
Then for
  any $\mu\in (0,\infty)$,
$$\| (\mu I + T_{t,h})^{-1} \eta_h\|_{\dH{-b}_h } \le  \mu^{(a+b)/2-1}
\|\eta_h\|_{\dH{a}_h},\Forall \eta_h\in \vhmt.$$
\end{lemma}

\subsection{Consistency}

The next step is to estimate the consistency error between 
$a^{k,M}(\cdot,\cdot)$ and $a_{h}^{k,M}(\cdot,\cdot)$ on $\mathbb V_h(D)$.
Its decay depends on a parameter $\beta \in (s,3/2)$, which will be related later to the regularity of the solution $u$ to \eqref{mainP}.

\begin{theorem}[Finite Element Consistency]\label{t:discrete_consistency}
Let $\beta\in(s,3/2)$.
We assume that the quadrature parameters $N^-$ and $N^+$ are chosen according to \eqref{e:choiceN}. {There} exists a constant
$C$ independent  of $h$, $k$ and $M$ satisfying
\begin{equation}\label{h-estimate}\begin{aligned}
	|a^{k,M}(\eta_h,\theta_h)&-a^{k,M}_h(\eta_h,\theta_h)|  \\
	&\le
        C(1+\ln(h^{-1}))h^{\beta-s}\|\eta_h\|_{\widetilde
            H^\beta(D)}\|\theta_h\|_{\widetilde H^{s}(D)}
\end{aligned}
\end{equation}
for all $\eta_h, \theta_h\in \mathbb{V}_h {(D)}$.
\end{theorem}
\begin{proof}
In this proof, $C$ denotes a generic constant independent of $h$, $M$,
$k$ and $t$.  

 Fix $\eta_h \in \mathbb V_h(D)$ and denote by $\tilde \eta_h$ its extension by zero outside $D$.
We first observe that for $\theta_h\in \mathbb{V}_h(D)$ and $\tilde \theta_h$ its extension by zero outside $D$, we have
\begin{equation}\begin{aligned}
(w^M(\widetilde \eta_h,t_j),\theta_h)_D&= ( \pi_h
w^M(\widetilde \eta_h,t_j),\widetilde \theta_h)_\omt,
\end{aligned}
\end{equation}
where $\pi_h$ denotes the $L^2$ projection onto $\mathbb V_h(\Omega^M(t))$.
Using the above identity and recalling that $t_j=e^{-y_j/2}$, we obtain
\begin{equation*}
\begin{split}
a^{k,M}(\eta_h,\theta_h)-a^{k,M}_h(\eta_h,\theta_h) 
= \frac{c_s }{2} \underbrace{k\sum_{t_j\le \frac 12} e^{sy_j}
  ( \pi_hw^M(\widetilde \eta_h,t_j)-w^M_h(\widetilde
  \eta_h,t_j),\widetilde \theta_h)_{\Omega^M(t)}}_{=:E_1
} \\ + 
\frac{c_s }{2}\underbrace{k\sum_{t_j> \frac 12} e^{sy_j} (
  \pi_hw^M(\widetilde \eta_h,t_j)-w^M_h( \widetilde \eta_h,t_j),\widetilde \theta_h)_{\Omega^M(t)}}_{=: E_2}. 
\end{split}
\end{equation*}
We bound the two terms separately and start with the latter.   

$\boxed{1}$ In view of the definitions \eqref{wMt} of $w^M(t)$ and \eqref{e:wMh} of $w^M_h(t)$, we have
\begin{equation} \label{e:usev_v_h}
 \pi_h w^M(\widetilde \eta_h,t)-w^M_h(\widetilde \eta_h,t)
= \pi_h v^M(\widetilde \eta_h,t)-v^M_h(\widetilde \eta_h,t).
\end{equation}
We recall that $T_t = T_{\Omega^M(t)}$ and $S_t$ are defined by \eqref{e:Tomega} and \eqref{e:Sop} respectively.
Using these operators and the relations satisfied by  $v^M(t)$ and  $v_h^M(t)$ (see \eqref{vtM} and  \eqref{e:vhM}), we arrive at 
\begin{equation}
\begin{aligned}
\pi_h w^M(\widetilde \eta_h,t)&-w^M_h(\widetilde \eta_h,t)
 = [\tzer_{t,h} (\tzer_{t,h}+t^2I)^{-1}- \pi_h\tzer_t(\tzer_t+t^2I)^{-1}]\widetilde \eta_h\\
&  = t^{2}(\tzer_{h,t} +t^{-2}I)^{-1} \pi_h(\tzer_{t,h}-\tzer_t)(\tzer_t+t^{2}I)^{-1} \widetilde \eta_h.
\end{aligned}
\label{vdif}
\end{equation}
Thus,
\begin{align*}
\| \pi_h w^M(\widetilde \eta_h,t)&-w^M_h(\widetilde \eta_h,t)\|_\ltomt\\
 &\leq t^{2}\|(\tzer_{t,h}+t^{2}I)^{-1}
   \pi_h(\tzer_{t,h}-\tzer_t)(\tzer_t+t^{2}I)^{-1}\|\ \| \widetilde \eta_h \|_\ltomt\\
 &\le t^{2}\|(\tzer_{t,h}+t^{2}I)^{-1} \pi_h\|\
   \|\tzer_{t,h}-\tzer_t\|\ 
 \|(\tzer_t+t^{2}I)^{-1}\|\  \| \widetilde \eta_h \|_\ltomt.
\end{align*}
Here we have used $\|\cdot\|$ to denote the operator norm of operators from $L^2(\omt)$ to $L^2(\omt)$.
Combining
$$
\| (\tzer_{t,h}+t^2I)^{-1}\pi_h\|\le t^{-2},\quad 
\| (\tzer_t+t^{2}I)^{-1}\|
\leq t^{-2}
$$
and Lemma~\ref{tzer} gives
$$
\|\pi_h w^M(\widetilde \eta_h,t)-w^M_h(\widetilde \eta_h,t)\| \leq C t^{-2} h^2 \|\widetilde \eta_h\|_{L^2(\omt)}.
$$
Whence, 
$$
\begin{aligned}
|E_2| &\le Ch^2k\sum_{t_j> \frac 12} e^{(s+1)jk} \|   {\widetilde\eta_h} \|_{L^2(\OMt)} \|  {\widetilde\theta_h} \|_{L^2(\OMt)}\\
&\leq C h^2 \|  \eta_h \|_{L^2(D)} \| \theta_h \|_{L^2(D)}\left(k\sum_{jk<2\ln2} e^{(s+1)jk}\right)\le C h^2 \|  \eta_h \|_{L^2(D)} \| \theta_h \|_{L^2(D)}.
\end{aligned}
$$

\def\HHp#1{\dH {#1}}
$\boxed{2}$ We now focus on $E_1$ which requires a finer
analysis using 
intermediate spaces. Also, we argue differently for
$\beta\in(1,3/2)$ and for $\beta\in(s,1]$.   In either case, we define
  $$\epsilon:=\min\{1-s,1/\ln(1/h)\}$$
  and note that
  \begin{equation}
  \epsilon^{-1} \le c(1+\ln(1/h)) \quad\hbox{and}\quad
  h^{-\epsilon}\le c
\label{epsilon}
\end{equation}
with $c$ depending on $s$ but not $h$.

When $\beta\in(1,3/2)$, we invoke \eqref{e:usev_v_h} again to deduce
\begin{equation}
|E_1| \leq k\sum_{t_j\le\frac 12} e^{sy_j}\|\pi_h v^M(\widetilde
\eta_h,t_j)-v^M_h( \widetilde \eta_h,t_j) \|_{\dH {-s}_h } \| \widetilde\theta_h \|_{\dH
{s}_h}.
\label{e1est}
\end{equation}
We set $\mu(t) := t^{-2}-1$ and compute
\begin{equation}\label{e:three-parts}
\begin{aligned}
\pi_h v^M(\widetilde \eta_h,t)&-v^M_h(\widetilde \eta_h,t) =
                          t^{-2} [(I+\mu(t)T_{t,h})^{-1}T_{t,h}
                          -\pi_h T_t (I+\mu(t)T_t)^{-1}]\widetilde \eta_h\\
& = (t\mu(t))^{-2} (\mu(t)^{-1}I+T_{t,h})^{-1}\pi_h(T_{t,h}-T_t)(\mu(t)^{-1}I+T_t)^{-1}\widetilde \eta_h,
\end{aligned}
\end{equation}
which is now estimated in three parts.
Lemma~\ref{Tshift} guarantees that
$$
	\|(\mu(t)^{-1}I+T_t)^{-1}\|
	_{\dH \beta \rightarrow \dH {\beta-2}}\le 1,
$$
where we recall that $\dH s$ stands for $\dH s(\Omega^M(t))$.
For the second part, the error estimate \eqref{Tdifapp} with $1+r = \beta$ reads
$$
	\| T_{t,h}-T_t \|_{\dH {\beta-2} \rightarrow
          L^2(\Omega^M(t))}\leq Ch^{\beta}.
        $$
        We estimate the last term of the product in the right hand side of \eqref{e:three-parts} by
$$\begin{aligned}	\|(\mu(t)^{-1} &+T_{t,h})^{-1}
        \pi_h\|_{L^2(\Omega^M(t))\rightarrow \dH {-s}_h }\\
&\le C  \|(\mu(t)^{-1} +T_{t,h})^{-1}\|_{ \dH {s+\epsilon}_h\rightarrow  \dH {-s}_h}  \|
\pi_h\|_{L^2(\Omega^M(t))\rightarrow \dH {s+\epsilon}_h}.
\end{aligned}$$
Thus, Lemma~\ref{Tshifth}, the inverse estimate and \eqref{epsilon} yield
$$
\|(\mu(t)^{-1} +T_{t,h})^{-1} \pi_h \|_{L^2(\Omega^M(t))\rightarrow \dH
  {-s}_h }\le C h^{-s-\epsilon} t^{(2s+\epsilon-2)}\le C h^{-s}
t^{(2s+\epsilon-2)}.
$$

Note that for $t\in (0,1/2]$,  $0<t^2\leq \mu(t)^{-1} \leq \frac 4 3 t^2 \leq \frac 1 3$ so
that 
$$ (t\mu(t))^{-2}\le  \frac {16t^2} 9.
$$
Combining the above estimates with 
\eqref{e:three-parts}  gives
\begin{equation}
\| \pi_h v^M(\widetilde \eta_h,t)-v^M_h(\widetilde \eta_h,t)\|_{\dH{-s}_h} \leq
Ct^{2s+\epsilon}h^{\beta-s} \|\widetilde{\eta}_h\|_{\dH\beta},
\label{vmest}
\end{equation}

Since $t_j=e^{-y_j/2}$,
$$e^{sy_j} t_j^{2s+\epsilon}= e^{-\epsilon y_j/2} .$$
Estimates \eqref{e1est}, \eqref{vmest} and \eqref{dsplus} then yield
\begin{equation}\label{e1-sum-bound}\begin{aligned}
	|E_1| &\le Ch^{\beta-s}k\sum_{ky_j\ge 2\ln2} e^{-\epsilon y_j/2} 
	 \|\widetilde \eta_h\|_{\dH
           \beta}\|\widetilde\theta_h\|_{\dH{s}}\\
         &\le Ch^{\beta-s}\epsilon^{-1} 
	 \|\widetilde \eta_h\|_{\dH
           \beta}\|\widetilde\theta_h\|_{\dH{s}}.\end{aligned}
\end{equation}

\boxed{3}
We bound the norms on $\Omega^M(t)$ by norms on $D$ using \eqref{extenhs} with $r=s$ and Lemma~\ref{hdbound} to arrive at
$$
\begin{aligned}
	|E_1| &\le Ch^{\beta-s}\epsilon^{-1} \|\widetilde \eta_h\|_{H^\beta(\RRD)}\|\theta_h\|_{\dH {s}(D)}.
\end{aligned}
$$
Applying the norm  equivalence \eqref{e:dot-tilde} gives
\begin{equation}\label{e:e1:final}
|E_1| \leq C h^{\beta-s} \epsilon^{-1} \| \eta_h\|_{\widetilde H^\beta(D)}\|\theta_h\|_{\widetilde H^{s}(D)}.
\end{equation}

\boxed{4} When $\beta\in(s,1]$, we bound \eqref{e:three-parts} using different norms. In fact, we have
$$
	\|(\mu(t)^{-1}I+T_t)^{-1}\|
	_{\dH \beta \rightarrow \dH {-1}}\le t^{\beta-1}, 
	\quad 
		\| T_{t,h}-T_t \|_{\dH {-1} \rightarrow
          L^2(\Omega^M(t))}\leq Ch,
$$
and by Lemma~\ref{Tshifth},
$$\begin{aligned}
	\|(\mu(t)^{-1} +T_{t,h})^{-1}& \pi_h \|_{L^2(\Omega^M(t))\rightarrow 
	  \dH {-s}_h }\\
        &\le 	\|(\mu(t)^{-1} +T_{t,h})^{-1} \pi_h
      \|_{\dH {1-\beta+s+\epsilon} \rightarrow 
	\dH {-s}_h }
\|\pi_h\|_{L^2(\omt)\rightarrow \dH {1-\beta+s+\epsilon}}\\
&\le 
C h^{-1+\beta-s} t^{(2s+\epsilon-\beta-1)}.
\end{aligned}
$$
These estimates lead \eqref{vmest} and hence \eqref{e1-sum-bound}
as well when $\beta \in (s,1]$.  
The remainder of the proof is the same as in the case $\beta \in (1,3/2)$ except that the norm equivalence
\eqref{e:dot-tilde} is invoked in place of Lemma~\ref{hdbound}.

\boxed{5} The proof of the theorem is complete upon combining the estimates for $E_1$ and $E_2$.
\end{proof}

\subsection{Error Estimates}
Now that the consistency error between $a(\cdot,\cdot)$ and $a_h^{k,M}(\cdot,\cdot)$ is obtained, we can apply Strang's lemma to deduce the convergence of the approximation $u_h$ towards $u$ in the energy norm.
To achieve this, we need a result regarding 
the stability and approximability of the Scott-Zhang interpolant
  $\pi^{sz}_h$ \cite{SZ90} in the  fractional spaces $\widetilde
  H^\beta(D)$.

This is the subject of the next lemma. Its proof is somewhat technical and given in Appendix~\ref{a:lagrange}.

\begin{lemma}[Scott-Zhang Interpolant] \label{fe-interpolant} Let $\beta \in (1,3/2)$.   Then, there is a constant $C$ independent of $h$ such that
\begin{equation}\label{e:stability_sc}
\|\pi^{sz}_h v\|_{\widetilde H^\beta(D)} \le C \|v\|_{\widetilde H^\beta(D)} 
\end{equation}
 and for $s\in [0,1]$,
\begin{equation}\label{e:approx_sc}
\|\pi^{sz}_h v -v\|_{\widetilde H^s(D)}\le C h^{\beta-s} \|v\|_{\widetilde H^\beta(D)},
\end{equation}
for all  $v \in  \widetilde H^\beta(D)$.
\end{lemma}

We note that the above lemma holds for $\beta\in (0,1)$ and $s\in
  (0,\beta)$ provided that $\pi^{sz}_h$ is replaced by $\pi_h$, the $L^2$ projection onto $\mathbb V_h(D)$; see e.g.
  Lemma 5.1 of \cite{BP15}.  
In order to consider both case simultaneously in the following proof, we
set  $\Pi_h   = \pi_h $ 
when $\beta\in [0,1]$ and $\Pi_h =\pi^{sz}_h $ when $\beta\in (1,3/2)$.

\begin{theorem}\label{t:total_error}
Assume that the solution $u$ of \eqref{mainWP} belongs to $\widetilde H^\beta(D)$ for $\beta \in (s,3/2)$.
Let $\delta:=\min(2-s,\beta)$ be as in Theorem~\ref{t:quad_error}, $k$ be the quadrature spacing and
 $c_I$ be the inverse constant in \eqref{e:inverse}. We assume that the quadrature parameters 
 $N^-$ and $N^+$ are chosen according to \eqref{e:choiceN}.  Let $\gamma(k)$ be given by \eqref{e:gammak} 
 and assume that $k$ is chosen sufficiently small so that
$$
c_I \gamma(k) h^{s-\delta} < 1.
$$
Moreover,
let $u_h \in \mathbb V_h(D)$ be the solution of \eqref{finaldiscrete}.   
Then there is a constant $C$
  independent of $h$, $M$ and $k$ satisfying
\begin{equation}\label{energy-bound}
\|u-u_h\|_{\widetilde H^s(D)}\leq C(\gamma(k)+e^{-cM}+(1+\ln{(h^{-1})}) h^{\beta-s})\|u\|_{\widetilde H^\beta(D)}.
\end{equation}
\end{theorem}
\begin{proof}
In our context, the first Strang lemma (see e.g. Theorem 4.1.1 in \cite{ciarlet}) reads 
$$
\| u - u_h \|_{\tHs(D)} \le C\inf_{v_h \in \mathbb V_h {(D)}} \left(\| u - v_h \|_{\tHs(D)} + \sup_{w_h \in \vh {(D)}} \frac{|(a-a_{h}^{k,M})(v_h,w_h)|}{\| w_h \|_{\tHs(D)}}\right),
$$
where $C$ is a constant independent of $h$, $k$ and $M$.
From the consistency estimates \eqref{e:quad_cons}, \eqref{truncateb} and \eqref{h-estimate}, we deduce that
$$\begin{aligned}
\|u-u_h\|_{\widetilde H^s(D)}&\leq C\| u - \Pi_h u\|_{\tH^s(D)} \\
&\qquad + C(\gamma(k)+e^{-cM}+(1+\ln{(h^{-1})}) h^{\beta-s})\|\Pi_h u\|_{\widetilde H^\beta(D)}
\end{aligned}
$$
The desired estimate follows from the approximability and stability of $\Pi_h$.
\end{proof}

\section{Numerical implementation and results}\label{s:numerical}
In this section, we present detailed numerical implementation to solve the following model problems.

\subsection{Model Problems.}\label{s:analytical}

One of the difficulties in developing numerical approximation to
\eqref{mainWP} is that there are relatively few examples where analytical
solutions are available.  One exception is the case when $D$ is the
unit ball in $\RRD$.   In that case, the solution to the variational
problem
\begin{equation}
a(u,\phi) = (1,\phi)_D,\Forall \phi\in \tHs(D)
\label{var1}
\end{equation}
is radial and given by, (see \cite{DG13})
\begin{equation} 
u(x)= \frac {2^{-2s}\Gamma(d/2)} {\Gamma(d/2+s) \Gamma(1+s) }
(1-|x|^2)^s 
\label{anals}
\end{equation}

It is also possible to compute the right hand side corresponding to the
solution $u(x)=1-|x|^2$ in the unit ball. The corresponding 
right hand side can be derived by first computing  the Fourier transform
of $\widetilde u$, i.e.,
$$\cF(\widetilde u)=2J_2(|\zeta|)/|\zeta|^2,$$
where $J_n$ is the Bessel function of the first kind. 
When $0<s<1$, we obtain
\begin{equation}\label{rhs_smooth}
f(x)=\cF^{-1}(2|\zeta|^{2s-2}J_2(|\zeta|))=\frac{2^{2s}\Gamma(d/2+s)}{\Gamma(d/2)\Gamma(2-s)}
{}_2F_1\left(d/2+s, s-1,d/2,|x|^2\right) ,
\end{equation}
where ${}_2F_1$ is the Gaussian or ordinary hypergeometric function.

\begin{remark} [Smoothness]\label{singularf}
Even though the solution $u(x)=1-|x|^2$ is infinitely differentiable on the
unit ball, the right hand side $f$ has limited smoothness.   Note that
$f$ is the restriction of $(-\Delta)^s \widetilde u$ to the unit ball.
Now $\widetilde u\in H^{3/2-\epsilon}(\RRD)$ for $\epsilon>0$ but
is not in $H^{3/2}(\RRD)$.  This means that $(-\Delta)^s \widetilde u$
is only in $H^{3/2-2s-\epsilon}(\RRD)$ and hence $f$ is only
in  $H^{3/2-2s-\epsilon}(\Omega)$.  This is in agreement with the
singular behavior of $ {}_2F_1\left(d/2+s, s-1,d/2,t\right)$ at $t=1$ (see
\cite{NIST}, Section 15.4).
In fact,
$$\begin{aligned}
{}_2F_1\left(d/2+s, s-1,d/2,1\right)&=\frac {\Gamma(d/2+s)
  \Gamma(1-2s)}{\Gamma(d/2+1-s) \Gamma(-s)}\quad \hbox{when }0<s<1/2,\\
\lim_{t\rightarrow 1^-} \frac{{}_2F_1\left(d/2+s, s-1,d/2,t\right)}
{-\log(1-t)} &= \frac {\Gamma(d/2)}{\Gamma(-1/2)\Gamma(1/2)} \quad
\hbox{when }s=1/2,\\
\lim_{t\rightarrow 1^-} \frac {{}_2F_1\left(d/2+s, s-1,d/2,t\right)}
{(1-t)^{-2s+1}} &= \frac {\Gamma(d/2)}{\Gamma(-1/2)\Gamma(1/2)} \quad
\hbox{when }1/2<s<1.
\end{aligned}
$$
This implies that for $s\ge 1/2$, the trace on $|x|=1$ of $f(x)$ given by \eqref{rhs_smooth} fails to exist (as for generic  functions in
$H^{3/2-2s}(\RRD)$).
This singular behavior affects the convergence rate of the finite element
method  when the finite element data vector
is approximated using standard numerical quadrature (e.g. Gaussian quadrature).   
\end{remark}

\subsection{Numerical Implementation} \label{ss:num_impl}
Based on the notations in Section~\ref{s:truncated}, 
we set $\Omega=D$ to be either the unit disk in $\RR^2$ or $D=(-1,1)$ in $\RR$. Let $\Omega^M(t)$ be corresponding dilated domains.
In one dimensional case, we consider $\mathcal T_h(D)$ to be a uniform mesh and $\vh(D)$ to be the continuous piecewise linear 
finite element space.
For the two dimensional case, $\mathcal T_h(D)$ a regular (in the sense of page 247 in \cite{ciarlet}) subdivision made of quadrilaterals.
In this case, $\vh(D)$ is the set of continuous piecewise bilinear functions.

\subsubsection*{Non-uniform Meshes for $\omt$} 
We extend $\mathcal T_h(D)$ to non-uniform meshes $\mathcal T^M_h(t)$, thereby violating the quasi-uniform assumption.
For $t\le1$, we use a quasi-uniform mesh on $\Omega^M(t)=\Omega^M(1)$ with the same mesh size $h$. 
When $t>1$ and $D=(-1,1)$, we use
an exponentially graded mesh outside of $D$,
i.e. the mesh points are $\pm e^{ih_0}$ for $i=1,\ldots,\lceil M/h\rceil$ with $h_0=h(\ln{\gamma})/M$,
where $\gamma$ is the radius of $\omt$ (see \eqref{e:omega_Mt}).
Therefore, we maintain the same number of mesh points for all $\Omega^M(t)$. 
When $D$ is a unit disk in $\RR^2$, we start with a coarse subdivision of $\omt$ as in the left of Figure~\ref{f:grid} (the coarse mesh of $D$ in grey). Note that all vertices of a square have the same radial coordinates. 
We also point out that the position of the vertices along the radial direction and outside of $D$ follow the same exponential 
distribution as in the one dimensional case. 
Then we refine each cell in $D$ by connecting the midpoints between opposite edges.
For the cells outside of $D$, we consider the same refinement 
in the polar coordinate system $(\ln r,\theta)$ with $r>1$ and $\theta\in[0,2\pi]$.
This guarantees that mesh points on the same radial direction still follows the exponential distribution after global refinements and the number of mesh points in $\mathcal T^M_h(t)$ is unchanged for all $t>0$. The figure on the right of Figure~\ref{f:grid} shows the 
exponentially graded mesh after three times global refinement.

\begin{figure}[ht!]
\begin{center}
    \begin{tabular}{cc}
    \includegraphics[scale=.45]{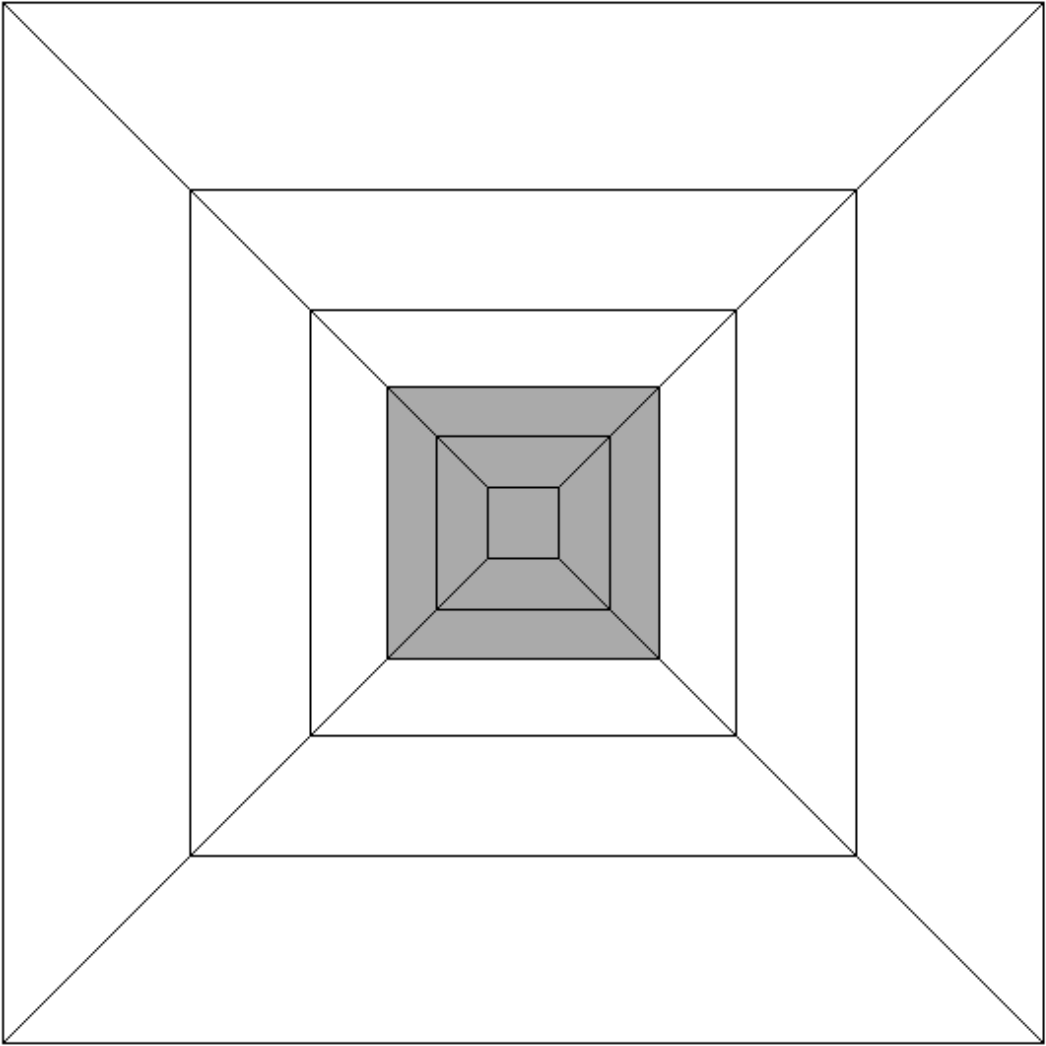} & \includegraphics[scale=.32]{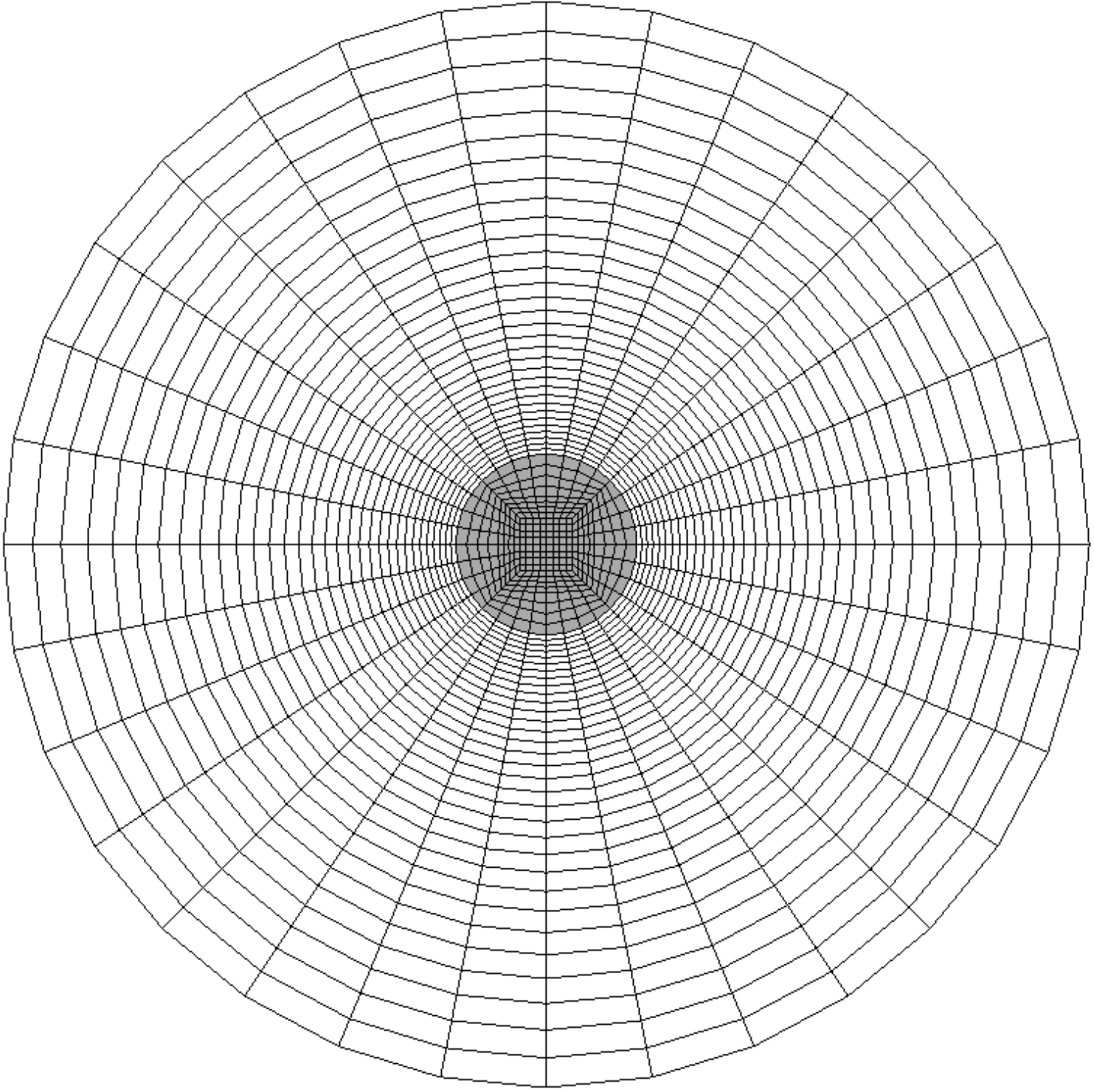}\\
    \end{tabular}
\end{center}
\caption{Coarse gird (left) and three-times-refined non-uniform 
grid (right) of $\Omega^M(t)$ with $M=4$ and $t=1$. Grids of $D$ are in grey.}
\label{f:grid}
\end{figure}

\subsubsection*{Matrix Aspects}
To express the linear system to be solved, we denote by $U$ to be the coefficient vector of $u_h$ and $F$ to be the coefficient vector of the $L^2$ 
projection of $f$ onto $\mathbb{V}_h(D)$.
Let $M_h(t)$ and $A_h(t)$ be the mass and stiffness matrix in $\mathbb{V}^M_h(t)$.
Denote $M_{D,h}$ to be the mass matrix in $\vh(D)$. The linear system is given by
\begin{equation}\label{system_matrix}
\frac{\sin{(\pi \beta)k}}{\pi}\sum_{i=-\Nminus}^{\Nplus}e^{s y_i}M_{D,h}(e^{y_i}M_h(t_i)+A_{h}(t_i))^{-1}A_{h}(t_i)U=F 
\end{equation}
with $y_i=ik$ and $t_i=e^{-y_i/2}$.
Here ${M_{D,h}},\ U$ and $F$ are all extended by zeros so that the dimension 
of the system is equal to the dimension of $\mathbb{V}^M_h(t)$.

\subsubsection*{Preconditioner}
Since the linear system is symmetric, we apply the Conjugate Gradient method 
to solve the above linear system. Due to the norm equivalence between 
$( L^2(D),H^1_0(D))_{s,2}$ and $\widetilde{H}^s(D)$, the condition number of the system matrix 
is bounded by $Ch^{-2s}$. In order to reduce the number of iterations in one dimensional space, 
we use fractional powers of the discrete Laplacian $L_{D,h}$ 
as a preconditioner, where $L_{D,h} :H^1_0(\Omega)\to L^2(D)$ is defined by
$$d_D(L_{D,h} w,\phi_h)=d_D(w,\phi_h), \Forall \phi_h\in \vh(D) .$$ 
This can be computed by the discrete sine transform similar to the implementation 
discussed in \cite{BLP16}. 
More precisely, the matrix representation of 
$L_{D,h}$ is given by ${(M_{D,h})}^{-1}{A_{D,h}}$, where $A_{D,h}$ is stiffness matrix in $\vh(D)$. 
The eigenvalues of $A_{D,h}$ and $M_{D,h}$ (for the same eigenvectors) are $a_j:=(2+\cos(j\pi h))/h$ and $m_j:=h(4+2\cos(j\pi h))/6$ for $j=1,\ldots,\dim(\vh(D))$, respectively.
Therefore, the eigenvalues of $L_h$ are given by $\lambda_{j,h}:=a_j/m_j$.
We use 
$$
B:= S\Lambda S
$$
as a preconditioner, where $S_{ij}:= \sqrt{ 2 h}\sin({ij \pi h})$ and $\Lambda$ is the diagonal matrix whose diagonal entries are $\lambda_{j,h}^{-s}/m_j$. We also note that $S^{-1}=S$.

In two dimensional space, we use the multilevel preconditioner advocated in \cite{bramble2000computational}.

\subsection{Numerical Illustration for the Non-smooth Solution}
We first consider the numerical experiments for the model problem \eqref{var1}
and study the behavior of the  $L^2(D)$ error.

\subsubsection*{Influence from the Sinc Quadrature and Domain Truncation.} 
When $D=(-1,1)$, we approximate 
the solution on the fixed uniform mesh with the mesh size $h=1/8192$. The domain truncation parameter $M$ is also fixed 
to be $20$. Thus, $h$ is small enough and $M$ is large enough so that the $L^2(D)$-error 
is dominant by the sinc quadrature spacing $k$. The left part of 
Figure~\ref{f:quad_domain} shows that the $L^2(D)$-error quickly converges
to the error dominant by the Galerkin approximation when $k$ approaches zero. Similar results are observed from
the right part of Figure~\ref{f:quad_domain} when the domain truncation parameter $M$ increases. In this case,
the mesh size $h=1/8192$ and the quadrature step size $k=0.2$.

\subsubsection*{Error Convergence from the Finite Element Approximation}
We note that we implement the numerical algorithm for the two dimensional case using the deal.ii Library 
\cite{bangerth2007deal} and we invert matrices in \eqref{system_matrix} using the direct solver from UMFPACK \cite{davis2007umfpack}. 
Figure~\ref{f:sol} shows  the approximated solutions for $s=0.3$ and $s=0.7$, respectively.
Table~\ref{table:l2err} reports errors $\|u-u_h\|_{L^2(D)}$ and rates of convergence with $s=0.3, 0.5$ and $0.7$. Here the quadrature spacing
($k=0.25$) and the domain truncation parameter ($M=4$) are fixed so that the finite element discretization dominates the error.

We note that 
Theorem 7.1 together with Theorem 5.4 in \cite{GG15} (see also Proposition 2.7 
in \cite{BPM17}) guarantees that when $\partial D$ is of class $C^\infty$ and  $f$ is
in $L^2(D)$, the solution of \eqref{mainWP} is in
$\widetilde H^{s+\alpha^-}(D)$ where 
\begin{equation}\label{e:alpha}
\alpha:=\min\{s,1/2\}
\end{equation} 
and $\alpha^-$ denotes any number strictly smaller that $\alpha$.
This indicates that the expected rate of convergence in $L^2(D)$ norm should be $\beta+\alpha^--s$
if the solution $u$ is in $\widetilde H^\beta(D)$.
Since the solution $u$ is in $H^{s+1/2-\epsilon}(D)$ (see \cite{AB15} for a proof),
Table~\ref{table:l2err} matches the expected rate of convergence $\min(1,s+1/2)$.

\begin{figure}[ht!]
\begin{center}
    \begin{tabular}{cc}
    \includegraphics[scale=.48]{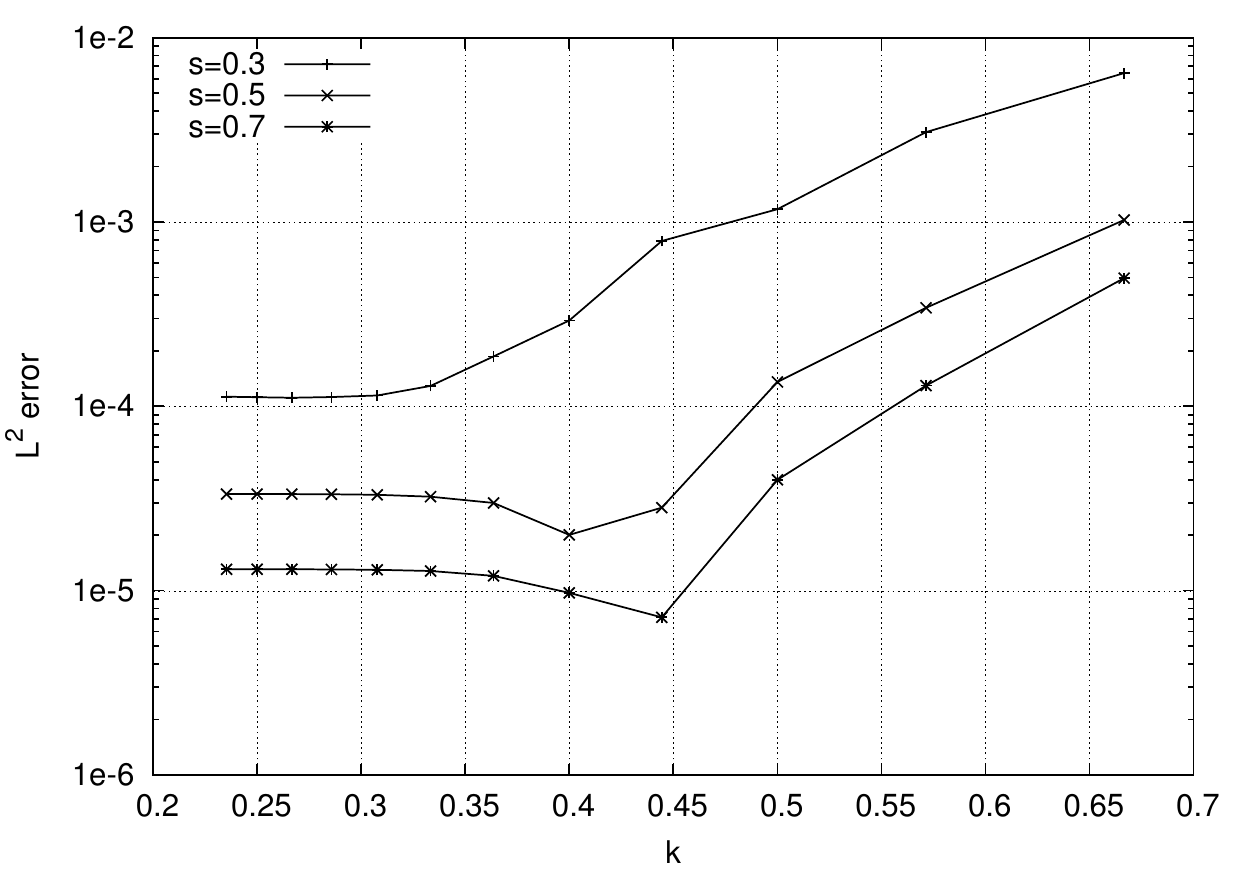} & \includegraphics[scale=.48]{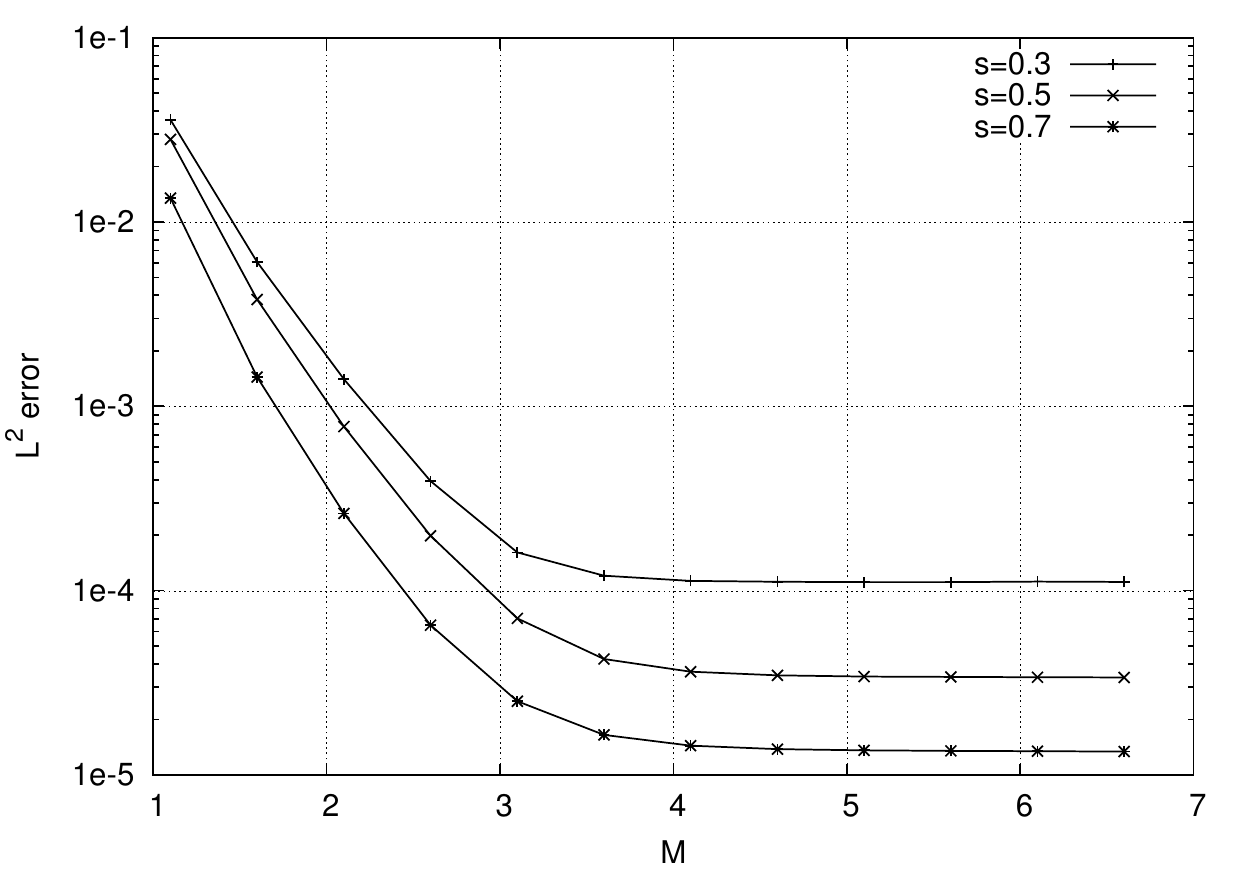}\\
    \end{tabular}
\end{center}
\caption{The above figures report the  $L^2(D)$-error behavior when $D=(-1,1)$.
The left one shows the error as a function of the quadrature spacing
$k$ for a fixed mesh size ($h=1/8192$) and domain truncation parameter ($M=20$). 
The right plot reports the  error as a function of the domain truncation parameter
$M$ with fixed mesh size ($h=1/8192$) and quadrature spacing
$(k=0.2)$.  The spatial error dominates when $k$ is small (left)
 and 
$M$  is large (right).}
\label{f:quad_domain}
\end{figure}

\begin{table}[h]
\begin{tabular}{|l|  l |l| l |l| l |l| l l}
\hline
$\#$DOFS & \multicolumn{2}{c|}{$s=0.3$}& \multicolumn{2}{c|}{$s=0.5$} & \multicolumn{2}{c|}{$s=0.7$} \\
\hline
345 & $2.69\times 10^{-1} $ & - & $1.63\times 10^{-1} $ & - & $1.03\times 10^{-1} $ & - \\\hline
1361  & $1.59\times 10^{-1} $ & 0.7575& $9.07\times 10^{-2} $ & 0.8426 & $5.55\times 10^{-2} $ & 0.8918 \\\hline
5409  & $9.56\times 10^{-2} $ & 0.7323& $5.05\times 10^{-2} $ & 0.8438 & $2.95\times 10^{-2} $ & 0.9091 \\\hline
21569  & $5.71\times 10^{-2} $ & 0.7447& $2.78\times 10^{-2} $ & 0.8633 & $1.54\times 10^{-2} $ & 0.9366 \\\hline
86145  & $3.38\times 10^{-2} $ & 0.7547& $1.51\times 10^{-2} $ & 0.8832 & $7.91\times 10^{-3} $ & 0.9641 \\\hline
344321  & $1.99\times 10^{-2} $ & 0.7644& $8.07\times 10^{-3} $ & 0.9004 & $3.97\times 10^{-3} $ & 0.9936 \\
\hline
\end{tabular}
\caption{$L^2(D)$-errors for different values of $s$ versus the number
  of degree of freedom used for the 2-D nonsmooth computations.
$\#$DOFS denotes the dimension of the finite element space $\vhmt.$}
\label{table:l2err}
\end{table}

\begin{figure}[ht!]
\begin{center}
    \begin{tabular}{cc}
    \includegraphics[scale=.25]{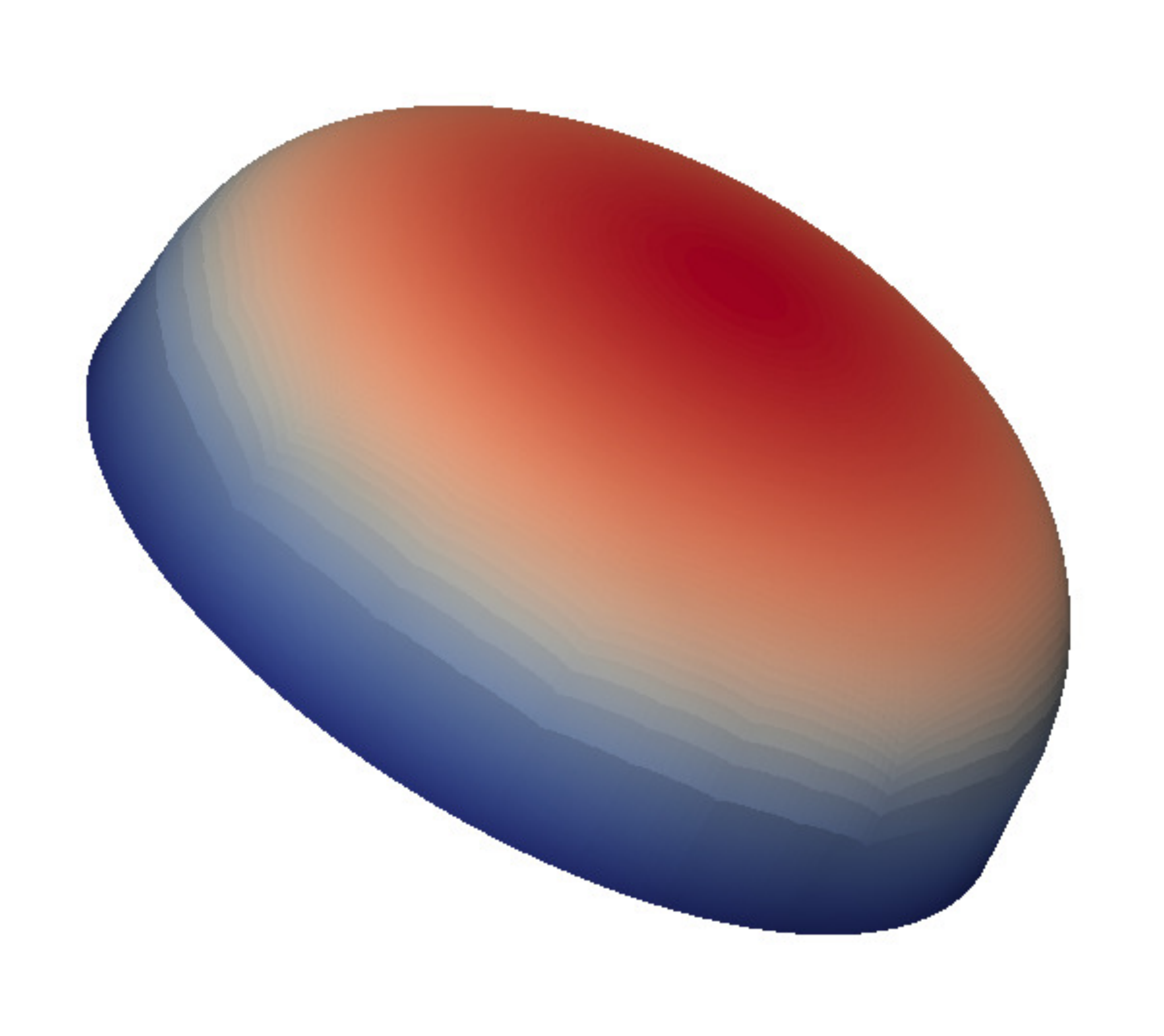} & \includegraphics[scale=.25]{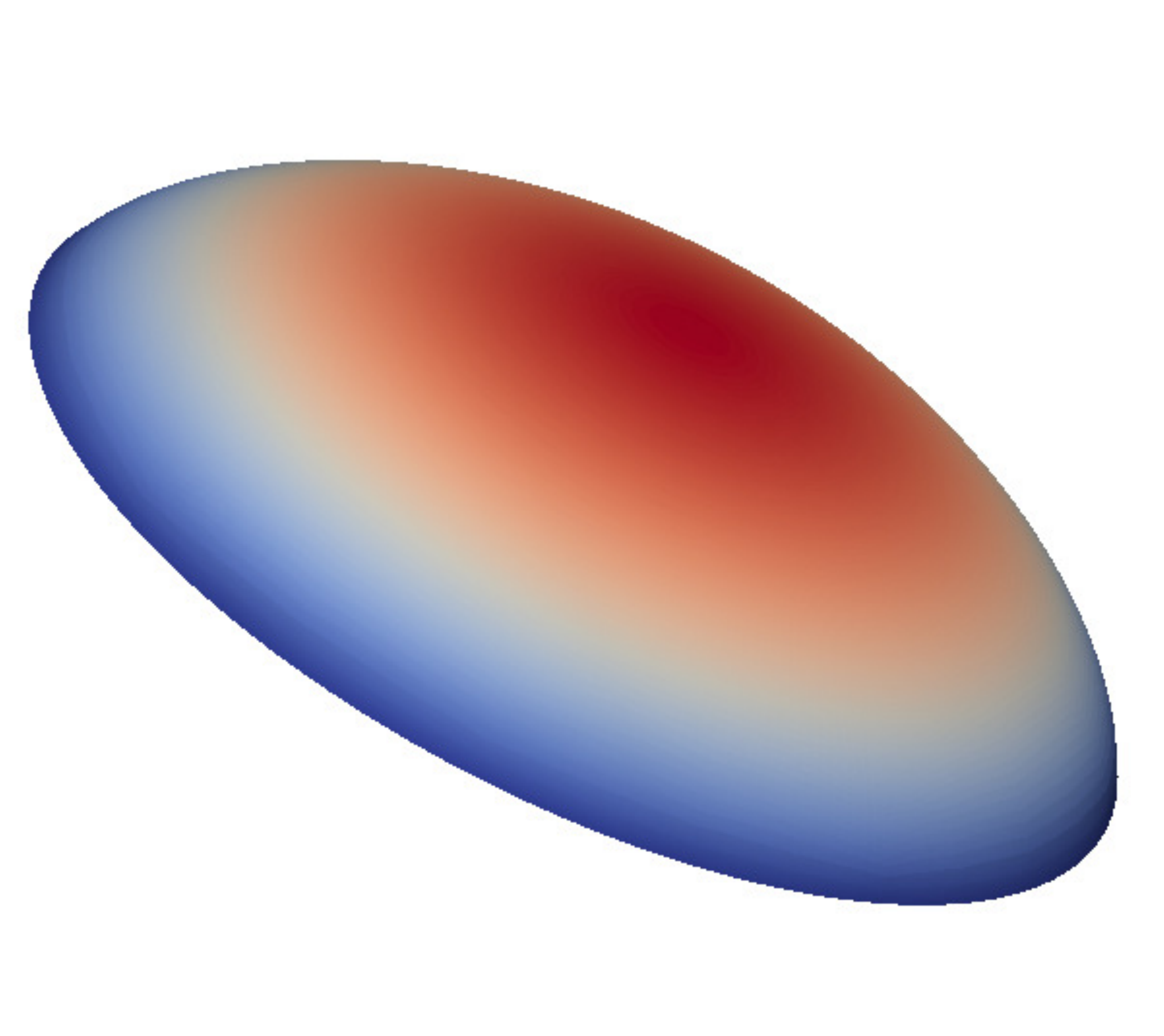}\\
    \end{tabular}
\end{center}
\caption{Approximated solutions of \eqref{anals} for $s=0.3$ (left) and $s=0.7$ (right) on the unit disk.}
\label{f:sol}
\end{figure}

\subsection{Numerical Illustration for the Smooth Solution}
When the solution is smooth, the finite element error (assuming
the exact computation of the stiffness entries, i.e. no consistency error) satisfies
$$
\|u-u_h\|_{L^2(D)}\le  c h^{2-s+\alpha^-},
$$
where $\alpha$ is given by \eqref{e:alpha}.
 In contrast, because of the inherent consistency error,  our method
 only guarantees (c.f.,
Theorem~\ref{t:total_error})   
\begin{equation}
\|u-u_h\|_{L^2(D)}\le  c h^{3/2-s+\alpha^-} .
\label{expected}
\end{equation}

Table~\ref{t:l2err_dealii} reports $L^2(D)$-errors and rates for the problem \eqref{mainWP} with the smooth solution $u(x)=1-|x|^2$ and
the corresponding right hand side data \eqref{rhs_smooth} in the unit disk. To see the error decay, here we choose the quadrature step
size $k=0.2$ and the domain truncation parameter $M=5$.   The observed decay in the error does not match the expected rate \eqref{expected}.
We think this loss of accuracy may be due either to the deterioration of the shape regularity constant in generating the subdivisions of $\Omega^M(t)$ (see Section~\ref{ss:num_impl}) or to the imprecise numerical integration of the singular right hand side in \eqref{rhs_smooth}.

To illustrate this, we
consider the one dimensional problem.   Instead of using
\eqref{rhs_smooth} to compute the right hand side vector, 
similar to \eqref{onedstiff}, we
compute
\begin{equation} (f,\phi_j)= a(u,\phi_j) = \frac{(\partial_L^{2s-1}\phi_j,u^\prime)_D
+(\partial_L^{2s-1}u,\phi_j^\prime)_D}{2\, \cos(s\pi)} \label{onedRHS}
\end{equation}
with $D=(-1,1)$. We note that when $s< 1/2$, the fractional derivative with the negative power $2s-1$ still makes sense for the local basis function $\phi_j$. 
The right hand side of \eqref{onedRHS} can now be computed exactly.  

\begin{table}[H]
\centering
\begin{tabular}{|l|  l |l| l |l| l |l| l l}
\hline
$\#$DOFS & \multicolumn{2}{c|}{$s=0.3$}& \multicolumn{2}{c|}{$s=0.5$} & \multicolumn{2}{c|}{$s=0.7$} \\
\hline
409 & $6.24\times 10^{-2} $ & - & $9.55\times 10^{-2} $ & - & $1.35\times 10^{-1} $ & - \\\hline
1617  & $2.90\times 10^{-2} $ & 1.10& $4.33\times 10^{-2} $ & 1.14 & $6.27\times 10^{-2} $ & 1.10 \\\hline
6433  & $1.44\times 10^{-2} $ &  1.01& $1.94\times 10^{-2} $ & 1.15 & $2.81\times 10^{-2} $ & 1.16 \\\hline
25665  & $7.21\times 10^{-3} $ & 1.00& $8.55\times 10^{-3} $ & 1.19 & $1.20\times 10^{-2} $ & 1.23 \\\hline
102529  & $3.56\times 10^{-3} $ & 1.02& $3.67\times 10^{-3} $ & 1.22 & $4.78\times 10^{-3} $ & 1.32 \\\hline
409857  & $1.74\times 10^{-3} $ & 1.04& $1.54\times 10^{-3} $ & 1.25 & $1.73\times 10^{-3} $ & 1.47 \\
\hline
\end{tabular}
\caption{$L^2(D)$-errors and rates for $s=0.3$, $0.5$ and $0.7$ for the problem
\eqref{mainWP} with the right hand side \eqref{rhs_smooth}. 
$\#$DOFS denotes the number of degree of freedoms of $\OMt$.}
\label{t:l2err_dealii}
\end{table}

We illustrate the convergence rate for the one dimensional case in Table~\ref{oned1} 
when the $L^2(D)$-projection of right hand side is computed from \eqref{onedRHS}.
In this case, we compute at $s=0.3,0.4,0.7$ as the
expression in \eqref{onedRHS} is not valid for $s=0.5$.  We also fix $k=0.2$ and $M=6$. In all cases,
we observe the predicted rate of convergence $\min(3/2,2-s)$, see \eqref{expected}. 

\begin{table}[H]
\centering
\begin{tabular}{|l|  l |l| l |l| l |l| l l}
\hline
$h$ & \multicolumn{2}{c|}{$s=0.3$}& \multicolumn{2}{c|}{$s=0.4$} & \multicolumn{2}{c|}{$s=0.7$} \\
\hline
1/16  & $4.51\times 10^{-4} $ & & $3.47\times 10^{-4} $ & & $9.27\times 10^{-4} $ & \\\hline
1/32  & $1.42\times 10^{-4} $ & 1.58& $1.02\times 10^{-4} $ & 1.77 &$4.16\times 10^{-4} $ & 1.16 \\\hline
1/64  & $4.25\times 10^{-5} $ & 1.63& $3.31\times 10^{-5} $ & 1.62 & $1.80\times 10^{-4} $ & 1.21 \\\hline
1/128  & $1.34\times 10^{-5} $ & 1.66& $1.14\times 10^{-5} $ & 1.54 & $7.66\times 10^{-5} $ & 1.23 \\\hline
1/256  & $4.43\times 10^{-6} $ & 1.59& $4.06\times 10^{-6} $ & 1.49 & $3.21\times 10^{-5} $ & 1.25 \\\hline
1/512  & $1.50\times 10^{-6} $ & 1.56& $1.46\times 10^{-6} $ & 1.48 & $1.33\times 10^{-5} $ & 1.27 \\\hline
\end{tabular}
\caption{$L^2(D)$-errors and rates for $s=0.3$, $0.4$ and $0.7$ for the
  one dimensional problem when right hand
  side of the discrete problem is computed by \eqref{onedRHS}.}
\label{oned1}
\end{table}


\appendix
\section{Proof of Lemma~\ref{fe-interpolant}}\label{a:lagrange}   

The proof of Lemma~\ref{fe-interpolant} requires the
following auxiliary localization result. We refer to \cite{faermann} for a similar result
in two dimensional space.

\begin{lemma} \label{rbound} For $r\in (0,1/2)$, let $v$ be in $H^r(D)$ and $\widetilde
v$ denote the extension by zero of $u$ to $\RRD$.  There exists a constant $C$ independent of $h$ such that 
$$\|\widetilde v\|_{H^r(\RRD)}^2\le C \bigg (h^{-2r} \|v\|_{L^2(D)}^2 +
\sum_{\tau\in \Thd} |v|_{H^r(\tau)}^2\bigg )
$$
with a constant $C$ independent of $h$.
\end{lemma}

\begin{proof} 
Let $\Tht$ be any quasi-uniform mesh (satisfying \eqref{a:shape-regular}
and \eqref{a:quasi-uniform}) which extends $\Thd$ beyond a unit
size neighborhood of $D$. 
Fix $\delta>0$ and for 
  $\tau \in \Thd $ set 
$$
\ttau=\cup_{\{\eta\in \Tht\  :\ \hbox{dist}(\eta,\tau)<\delta h\}} \eta.  
$$
Let
$$D_h^\delta=\mathop{\bigcup}_{\tau\in \Thd} \ttau$$
and let $\Thdd$ denote the set of $\tau\in \Tht$ contained in $D_h^\delta$.
Finally, for $\tau\in \Thdd\setminus \Thd$, set 
$$
\ttau=\bigcup_{\{\eta\in \Thdd\  :\ 
\hbox{dist}(\eta,\tau)<\delta h\}}\eta .
$$

Fix $v\in H^r(D)$.  Since $\widetilde v$ vanishes outside of $D_h^\delta$,
$$\begin{aligned} |\widetilde v|_{H^r(\RRD)}^2&=\int_{D_h^\delta}\int_{D_h^\delta}
\frac {(\widetilde v(x)-\widetilde v(y))^2}
{|x-y|^{d+2r}} \, dx \, dy \\
&\qquad +2\int_{D}\int_{(D_h^\delta)^c}
\frac {v(y)^2}
{|x-y|^{d+2r}} \, dx \, dy=:J_1+J_2.
\end{aligned}$$

The second integral above is bounded by
\begin{equation}
\begin{aligned} J_2 &\le 2 \int_{D}\int_{|x-y|\ge \delta h}
\frac {v(y)^2}
{|x-y|^{d+2r}} \, dx \, dy\\
&=2 (\delta h)^{-2r}  \int_{D} v(y)^2 \, dy 
\int_{|z|\ge 1} |z|^{-d-2r} \, dz= C h^{-2r}\|v\|_{L^2(D)}^2.
\end{aligned}
\label{j2bound}
\end{equation}

Expanding the first integral gives
\begin{equation}\begin{aligned}
J_1&= \sum_{\tau\in \Thdd} \int_\tau  \int_\ttau 
\frac {(v(x)-v(y))^2}
{|x-y|^{d+2r}} \, dx \, dy\\
&\qquad  +\sum_{\tau\in \Thdd} \int_\tau \int_{\ttau^c} \frac {(v(x)-v(y))^2}
{|x-y|^{d+2r}}\, dx \, dy=:J_3+J_4.\end{aligned}
\label{j1exp}
\end{equation}
Applying the arithmetic-geometric mean inequality gives
\begin{equation}
J_4\le 2 \sum_{\tau\in \Thdd} \int_\tau \int_{\ttau^c}  \frac
{v(x)^2+v(y)^2}{|x-y|^{d+2r}}\, dx \, dy.
\label{partj1}
\end{equation}
As in \eqref{j2bound},
\begin{equation}
J_5:= \sum_{\tau\in \Thdd} \int_\tau \int_{\ttau^c} \frac {v(y)^2} {|x-y|^{d+2r}}\
 \, dx \, dy 
\le C h^{-2r} \|v\|_{L^2(D)}^2.
\label{firstsum}
\end{equation}
Now,
$$\begin{aligned} 
\{(\tau,\tau_1)&\ : \ \tau\in \Thdd \hbox{ and } \tau_1\in \ttau^c\}\\
&=\{(\tau,\tau_1)\in \Thdd\times \Thdd\ : \
\hbox{dist}(\tau,\tau_1)>\delta h\}\\
&=\{(\tau,\tau_1)\ : \ \tau_1\in \Thdd \hbox{ and } \tau\in
\ttau_1^c\}.\end{aligned}
$$
Using this and Fubini's Theorem gives
$$\begin{aligned}
 \sum_{\tau\in \Thdd}& \int_\tau \int_{\ttau^c} \frac {v(x)^2} {|x-y|^{d+2r}}\
 \, dx \, dy\\
&\qquad  = \sum_{\tau_1\in \Thdd} \int_{\tau_1} \int_{\ttau_1^c} \frac {v(x)^2} {|x-y|^{d+2r}}\
 \, dy \, dx= J_5.
\end{aligned}
$$
Thus $J_4\le 4 J_5$ and is bounded by the right hand side of
\eqref{firstsum}.

For $J_3$, we clearly have
$$J_3\le \sum_{\tau\in \Thdd} |v|_{H^r(\ttau)}^2.$$
For any element $\tau' \in \Thdd$, let $v_{\tau'}$ denote $\widetilde v$ restricted to $\tau'$
and  extended by zero
outside.  
As $r\in (0,1/2)$, 
$v_{\tau'} \in H^r(\RRD)$ and satisfies
$$ \|v_{\tau'}\|_{H^r(\RRD)}\le C \|v\|_{H^r(\tau')}.$$
The constant $C$ above only depends on Lipschitz constants associated
with $\tau'$ (see \cite{grisvard,DD12}), which in turn only depend on
the constants appearing in \eqref{a:shape-regular}.  We use the triangle inequality to get
$$|v|_{H^r(\ttau)} \le  \sum_{\tau^\prime\subset \ttau} |v_{\tau'}|_{H^r(\ttau)}$$
and hence a Cauchy-Schwarz inequality implies that
$$|v|_{H^r(\ttau)}^2 \le  N_\tau \sum_{\tau' \subset \ttau}
|v_{\tau'}|^2_{H^r(\ttau)}\le C \sum_{\tau' \subset \ttau}
|v_{\tau'}|^2_{H^r(\tau^\prime)}
$$
with $N_\tau$ denoting the number of elements in $\ttau$.   As the mesh
is quasi-uniform, $N_\tau$ can be bounded independently of $h$.  
In
addition, the mesh quasi-uniformity condition also implies that each $\tau' \in
\Thd$ is contained in a most a fixed number (independent of $h$) of
$\ttau$ (with $\tau\in\Thdd$).  Thus,
$$J_3
\le C\sum_{\tau^\prime\in \Thd} \|v\|^2_{H^r(\tau^\prime)}.
$$
Combining the estimates for $J_2,J_3$ and $J_4$ completes the proof of
the lemma.
\end{proof}

\begin{proof} [Proof of Lemma~\ref{fe-interpolant}]  
In this proof, $C$ denotes a generic constant independent of $h$ and $j$ defined later.
The inequality  (4.1) of \cite{SZ90} guarantees that for $\tau \in \cT_h$, we have
\begin{equation}
\|v-\pi^{sz}_hv\|_{H^{m}(\tau)} \le C \sum_{k=0}^m
  h^{k-m} \|v-p\|_{H^k(S_\tau)} ,\qquad \hbox{ for }m=0,1,
\label{sz-bound}
\end{equation}
for any linear polynomial $p$ and $v\in H^{1}(S_\tau)$.    Here
$S_\tau$ denotes the union of $\tau^\prime\in \cT_h$ with
$\tau\cap\tau^\prime \neq \emptyset$.  

Now, we map $\tau$ to the reference element using an affine
transformation.   The mapping takes $S_\tau$ to $\widehat S_\tau$.
Our aim is to take advantage of the averaged Taylor polynomial constructed in \cite{dupont-scott}, which requires the domain to be star-shaped with respect to a ball (of uniform diameter). The patch $\widehat S_\tau$ may not satisfy this property.
Howecer, it can be written as the (overlapping) union  of domains $\widehat D_j$ with
each $ \widehat D_j$ consisting of the union of pairs of elements of $\widehat S_\tau$ sharing a common
face.  These $\widehat D_j$ are star-shaped with respect to balls of diameter depending on the shape regularity constant of the subdivision, which is uniform thanks to \eqref{a:shape-regular}.
Hence,  the averaged Taylor
polynomial $\cQ_j$ constructed in \cite{dupont-scott} satisfies (see
Theorem~6.1 of \cite{dupont-scott}),  for all
$v\in H^{\beta}(\widehat S_j)$,
\begin{equation}
\|v-\cQ_jv\|_{H^1(\widehat D_j)} 
\le C  |v|_{H^{\beta}(\widehat D_j)}.
\label{dupscot}
\end{equation}
Taking $\| \cdot \|_{\widehat D_j}$ to be $\|\cdot\|_{L^2(\widehat D_j)} $ or
$\|\cdot\|_{H^1(\widehat D_j)}$ and $|\cdot |_{\widehat D_j} = |\cdot|_{H^\beta(\widehat D_j)}$ in
Theorem~7.1  of \cite{dupont-scott}  implies that \eqref{dupscot} holds
with $\widehat D_j$ replaced by $\widehat S_j$.  
This, \eqref{sz-bound} and a
Bramble-Hilbert  argument
implies that for $v\in H^\beta(D)\cap H^1_0(D)$,
 \begin{equation}  \label{frac-est}
\|v-\pi_h^{sz} v\|_{L^2(D)} +  h \|v-\pi_h^{sz}v\|_{H^1(D)} 
\le C h^\beta |v|_{H^{\beta}(D)}.
\end{equation}
Inequality \eqref{e:approx_sc} follows from \eqref{frac-est} and
interpolation.

We cannot use Theorem~7.1 of \cite{dupont-scott} to derive 
\eqref{e:stability_sc} because of the non-locality of the norm
$|\cdot|_{H^\beta(D)}$.    Instead,
we apply Lemma~\ref{rbound},
\eqref{frac-est}, and the fact that $|\pi_h^{sz}v|_{H^\beta(\tau)}=0$ to
obtain, for $v\in H^\beta(D)\cap H^1_0(D)$,
\begin{equation}
\begin{aligned} |v-\pi_h^{sz}v |_{H^\beta (D)}^2&\le C \bigg (h^{2-2\beta} \|\nabla (v-\pi_h^{sz}v)\|_{L^2(D)}^2 +
\sum_{\tau\in \Thd} |v|_{H^r(\tau)}^2\bigg )\\
&\le |v|^2_{H^\beta(D)}.\end{aligned}
\label{seminb}
\end{equation}
The norms in  \eqref{e:stability_sc} can be replaced
by $\|\cdot \|_{H^\beta(D)}$ and hence \eqref{e:stability_sc} follows
from \eqref{frac-est} and \eqref{seminb}.
\end{proof}


\begin{thebibliography}{10}
\providecommand{\url}[1]{{#1}}
\providecommand{\urlprefix}{URL }
\expandafter\ifx\csname urlstyle\endcsname\relax
  \providecommand{\doi}[1]{DOI~\discretionary{}{}{}#1}\else
  \providecommand{\doi}{DOI~\discretionary{}{}{}\begingroup
  \urlstyle{rm}\Url}\fi

\bibitem{AB15}
Acosta, G., Borthagaray, J.P.: A fractional {L}aplace equation: regularity of
  solutions and finite element approximations.
\newblock SIAM J. Numer. Anal. \textbf{55}(2), 472--495 (2017).
\newblock \doi{10.1137/15M1033952}.
\newblock \urlprefix\url{https://doi.org/10.1137/15M1033952}

\bibitem{auscher2002solution}
Auscher, P., Hofmann, S., Lewis, J.L., Tchamitchian, P.: Extrapolation of
  {C}arleson measures and the analyticity of {K}ato's square-root operators.
\newblock Acta Math. \textbf{187}(2), 161--190 (2001).
\newblock \doi{10.1007/BF02392615}.
\newblock \urlprefix\url{https://doi.org/10.1007/BF02392615}

\bibitem{bacuta-thesis}
Bacuta, C.: Interpolation between subspaces of {H}ilbert spaces and
  applications to shift theorems for elliptic boundary value problems and
  finite element methods.
\newblock ProQuest LLC, Ann Arbor, MI (2000).
\newblock
  \urlprefix\url{http://gateway.proquest.com/openurl?url_ver=Z39.88-2004&rft_val_fmt=info:ofi/fmt:kev:mtx:dissertation&res_dat=xri:pqdiss&rft_dat=xri:pqdiss:9994203}.
\newblock Thesis (Ph.D.)--Texas A\&M University

\bibitem{bangerth2007deal}
Bangerth, W., Hartmann, R., Kanschat, G.: deal.{II}---a general-purpose
  object-oriented finite element library.
\newblock ACM Trans. Math. Software \textbf{33}(4), Art. 24, 27 (2007).
\newblock \doi{10.1145/1268776.1268779}.
\newblock \urlprefix\url{https://doi.org/10.1145/1268776.1268779}

\bibitem{BKW01}
Biler, P., Karch, G., Woyczy\'{n}ski, W.A.: Critical nonlinearity exponent and
  self-similar asymptotics for {L}\'{e}vy conservation laws.
\newblock Ann. Inst. H. Poincar\'{e} Anal. Non Lin\'{e}aire \textbf{18}(5),
  613--637 (2001).
\newblock \doi{10.1016/S0294-1449(01)00080-4}.
\newblock \urlprefix\url{https://doi.org/10.1016/S0294-1449(01)00080-4}

\bibitem{BLP18}
Bonito, A., Lei, W., Pasciak, J.E.: On sinc quadrature approximations of
  fractional powers of regularly accretive operators.
\newblock J. Numer. Math. To appear

\bibitem{BLP16}
Bonito, A., Lei, W., Pasciak, J.E.: The approximation of parabolic equations
  involving fractional powers of elliptic operators.
\newblock J. Comput. Appl. Math. \textbf{315}, 32--48 (2017).
\newblock \doi{10.1016/j.cam.2016.10.016}.
\newblock \urlprefix\url{https://doi.org/10.1016/j.cam.2016.10.016}

\bibitem{BP13}
Bonito, A., Pasciak, J.E.: Numerical approximation of fractional powers of
  elliptic operators.
\newblock Math. Comp. \textbf{84}(295), 2083--2110 (2015).
\newblock \doi{10.1090/S0025-5718-2015-02937-8}.
\newblock \urlprefix\url{https://doi.org/10.1090/S0025-5718-2015-02937-8}

\bibitem{BP15}
Bonito, A., Pasciak, J.E.: Numerical approximation of fractional powers of
  regularly accretive operators.
\newblock IMA J. Numer. Anal. \textbf{37}(3), 1245--1273 (2017).
\newblock \doi{10.1093/imanum/drw042}.
\newblock \urlprefix\url{https://doi.org/10.1093/imanum/drw042}

\bibitem{BPM17}
Borthagaray, J.P., Del~Pezzo, L.M., Mart\'{i}nez, S.: Finite element
  approximation for the fractional eigenvalue problem.
\newblock J. Sci. Comput. \textbf{77}(1), 308--329 (2018).
\newblock \doi{10.1007/s10915-018-0710-1}.
\newblock \urlprefix\url{https://doi.org/10.1007/s10915-018-0710-1}

\bibitem{bramble2000computational}
Bramble, J.H., Pasciak, J.E., Vassilevski, P.S.: Computational scales of
  {S}obolev norms with application to preconditioning.
\newblock Math. Comp. \textbf{69}(230), 463--480 (2000).
\newblock \doi{10.1090/S0025-5718-99-01106-0}.
\newblock \urlprefix\url{https://doi.org/10.1090/S0025-5718-99-01106-0}

\bibitem{BX91}
Bramble, J.H., Xu, J.: Some estimates for a weighted {$L^2$} projection.
\newblock Math. Comp. \textbf{56}(194), 463--476 (1991).
\newblock \doi{10.2307/2008391}.
\newblock \urlprefix\url{https://doi.org/10.2307/2008391}

\bibitem{00BZ}
Bramble, J.H., Zhang, X.: The analysis of multigrid methods.
\newblock In: Handbook of numerical analysis, {V}ol. {VII}, Handb. Numer.
  Anal., VII, pp. 173--415. North-Holland, Amsterdam (2000)

\bibitem{chandler}
Chandler-Wilde, S.N., Hewett, D.P., Moiola, A.: Interpolation of {H}ilbert and
  {S}obolev spaces: quantitative estimates and counterexamples.
\newblock Mathematika \textbf{61}(2), 414--443 (2015).
\newblock \doi{10.1112/S0025579314000278}.
\newblock \urlprefix\url{http://dx.doi.org/10.1112/S0025579314000278}

\bibitem{ciarlet}
Ciarlet, P.G.: The finite element method for elliptic problems, \emph{Classics
  in Applied Mathematics}, vol.~40.
\newblock Society for Industrial and Applied Mathematics (SIAM), Philadelphia,
  PA (2002).
\newblock \doi{10.1137/1.9780898719208}.
\newblock \urlprefix\url{https://doi.org/10.1137/1.9780898719208}.
\newblock Reprint of the 1978 original [North-Holland, Amsterdam; MR0520174 (58
  \#25001)]

\bibitem{Constantin02}
Constantin, P.: Energy spectrum of quasigeostrophic turbulence.
\newblock Physical review letters \textbf{89}(18), 184,501 (2002)

\bibitem{CCW01}
Constantin, P., Majda, A.J., Tabak, E.: Formation of strong fronts in the
  {$2$}-{D} quasigeostrophic thermal active scalar.
\newblock Nonlinearity \textbf{7}(6), 1495--1533 (1994).
\newblock \urlprefix\url{http://stacks.iop.org/0951-7715/7/1495}

\bibitem{Tankov03}
Cont, R., Tankov, P.: Financial modelling with jump processes.
\newblock Chapman \& Hall/CRC Financial Mathematics Series. Chapman \&
  Hall/CRC, Boca Raton, FL (2004)

\bibitem{CC03}
C\'{o}rdoba, A., C\'{o}rdoba, D.: A maximum principle applied to
  quasi-geostrophic equations.
\newblock Comm. Math. Phys. \textbf{249}(3), 511--528 (2004).
\newblock \doi{10.1007/s00220-004-1055-1}.
\newblock \urlprefix\url{https://doi.org/10.1007/s00220-004-1055-1}

\bibitem{davis2007umfpack}
Davis, T.A.: Umfpack version 5.2.0 user guide.
\newblock University of Florida  (2007)

\bibitem{DG13}
D'Elia, M., Gunzburger, M.: The fractional {L}aplacian operator on bounded
  domains as a special case of the nonlocal diffusion operator.
\newblock Comput. Math. Appl. \textbf{66}(7), 1245--1260 (2013).
\newblock \doi{10.1016/j.camwa.2013.07.022}.
\newblock \urlprefix\url{https://doi.org/10.1016/j.camwa.2013.07.022}

\bibitem{DD12}
Demengel, F., Demengel, G.: Functional spaces for the theory of elliptic
  partial differential equations.
\newblock Universitext. Springer, London; EDP Sciences, Les Ulis (2012).
\newblock \doi{10.1007/978-1-4471-2807-6}.
\newblock \urlprefix\url{https://doi.org/10.1007/978-1-4471-2807-6}.
\newblock Translated from the 2007 French original by Reinie Ern\'{e}

\bibitem{Droniou10}
Droniou, J.: A numerical method for fractal conservation laws.
\newblock Math. Comp. \textbf{79}(269), 95--124 (2010).
\newblock \doi{10.1090/S0025-5718-09-02293-5}.
\newblock \urlprefix\url{https://doi.org/10.1090/S0025-5718-09-02293-5}

\bibitem{DGLZ12}
Du, Q., Gunzburger, M., Lehoucq, R.B., Zhou, K.: Analysis and approximation of
  nonlocal diffusion problems with volume constraints.
\newblock SIAM Rev. \textbf{54}(4), 667--696 (2012).
\newblock \doi{10.1137/110833294}.
\newblock \urlprefix\url{https://doi.org/10.1137/110833294}

\bibitem{dupont-scott}
Dupont, T., Scott, R.: Polynomial approximation of functions in {S}obolev
  spaces.
\newblock Math. Comp. \textbf{34}(150), 441--463 (1980).
\newblock \doi{10.2307/2006095}.
\newblock \urlprefix\url{https://doi.org/10.2307/2006095}

\bibitem{faermann}
Faermann, B.: Localization of the {A}ronszajn-{S}lobodeckij norm and
  application to adaptive boundary element methods. {II}. {T}he
  three-dimensional case.
\newblock Numer. Math. \textbf{92}(3), 467--499 (2002).
\newblock \doi{10.1007/s002110100319}.
\newblock \urlprefix\url{https://doi.org/10.1007/s002110100319}

\bibitem{GH15}
Gatto, P., Hesthaven, J.S.: Numerical approximation of the fractional
  {L}aplacian via {$hp$}-finite elements, with an application to image
  denoising.
\newblock J. Sci. Comput. \textbf{65}(1), 249--270 (2015).
\newblock \doi{10.1007/s10915-014-9959-1}.
\newblock \urlprefix\url{https://doi.org/10.1007/s10915-014-9959-1}

\bibitem{grisvard}
Grisvard, P.: Elliptic problems in nonsmooth domains, \emph{Classics in Applied
  Mathematics}, vol.~69.
\newblock Society for Industrial and Applied Mathematics (SIAM), Philadelphia,
  PA (2011).
\newblock \doi{10.1137/1.9781611972030.ch1}.
\newblock \urlprefix\url{https://doi.org/10.1137/1.9781611972030.ch1}.
\newblock Reprint of the 1985 original [MR0775683], With a foreword by Susanne
  C. Brenner

\bibitem{GG15}
Grubb, G.: Fractional {L}aplacians on domains, a development of
  {H}\"{o}rmander's theory of {$\mu$}-transmission pseudodifferential
  operators.
\newblock Adv. Math. \textbf{268}, 478--528 (2015).
\newblock \doi{10.1016/j.aim.2014.09.018}.
\newblock \urlprefix\url{https://doi.org/10.1016/j.aim.2014.09.018}

\bibitem{HPGS95}
Held, I.M., Pierrehumbert, R.T., Garner, S.T., Swanson, K.L.: Surface
  quasi-geostrophic dynamics.
\newblock J. Fluid Mech. \textbf{282}, 1--20 (1995).
\newblock \doi{10.1017/S0022112095000012}.
\newblock \urlprefix\url{https://doi.org/10.1017/S0022112095000012}

\bibitem{JK95}
Jerison, D., Kenig, C.E.: The inhomogeneous {D}irichlet problem in {L}ipschitz
  domains.
\newblock J. Funct. Anal. \textbf{130}(1), 161--219 (1995).
\newblock \doi{10.1006/jfan.1995.1067}.
\newblock \urlprefix\url{https://doi.org/10.1006/jfan.1995.1067}

\bibitem{kato1961}
Kato, T.: Fractional powers of dissipative operators.
\newblock J. Math. Soc. Japan \textbf{13}, 246--274 (1961).
\newblock \doi{10.2969/jmsj/01330246}.
\newblock \urlprefix\url{https://doi.org/10.2969/jmsj/01330246}

\bibitem{KST06}
Kilbas, A.A., Srivastava, H.M., Trujillo, J.J.: Theory and applications of
  fractional differential equations, \emph{North-Holland Mathematics Studies},
  vol. 204.
\newblock Elsevier Science B.V., Amsterdam (2006)

\bibitem{Levendorskii04}
Levendorski\u{\i}, S.Z.: Pricing of the {A}merican put under {L}\'{e}vy
  processes.
\newblock Int. J. Theor. Appl. Finance \textbf{7}(3), 303--335 (2004).
\newblock \doi{10.1142/S0219024904002463}.
\newblock \urlprefix\url{https://doi.org/10.1142/S0219024904002463}

\bibitem{lions}
Lions, J.L., Magenes, E.: Non-homogeneous boundary value problems and
  applications. {V}ol. {II}.
\newblock Springer-Verlag, New York-Heidelberg (1972).
\newblock Translated from the French by P. Kenneth, Die Grundlehren der
  mathematischen Wissenschaften, Band 182

\bibitem{sinc_int}
Lund, J., Bowers, K.L.: Sinc methods for quadrature and differential equations.
\newblock Society for Industrial and Applied Mathematics (SIAM), Philadelphia,
  PA (1992).
\newblock \doi{10.1137/1.9781611971637}.
\newblock \urlprefix\url{https://doi.org/10.1137/1.9781611971637}

\bibitem{mclean}
McLean, W.: Strongly elliptic systems and boundary integral equations.
\newblock Cambridge University Press, Cambridge (2000)

\bibitem{mueller}
Mueller, C.: The heat equation with {L}\'{e}vy noise.
\newblock Stochastic Process. Appl. \textbf{74}(1), 67--82 (1998).
\newblock \doi{10.1016/S0304-4149(97)00120-8}.
\newblock \urlprefix\url{https://doi.org/10.1016/S0304-4149(97)00120-8}

\bibitem{NIST}
Olver, F., Lozier, D., Boisvert, R., Clark, C.: Nist digital library of
  mathematical functions.
\newblock Online companion to [65]: http://dlmf. nist. gov  (2010)

\bibitem{Pham97}
Pham, H.: Optimal stopping, free boundary, and {A}merican option in a
  jump-diffusion model.
\newblock Appl. Math. Optim. \textbf{35}(2), 145--164 (1997).
\newblock \doi{10.1007/s002459900042}.
\newblock \urlprefix\url{https://doi.org/10.1007/s002459900042}

\bibitem{ros-oton}
Ros-Oton, X., Serra, J.: The {D}irichlet problem for the fractional
  {L}aplacian: regularity up to the boundary.
\newblock J. Math. Pures Appl. (9) \textbf{101}(3), 275--302 (2014).
\newblock \doi{10.1016/j.matpur.2013.06.003}.
\newblock \urlprefix\url{https://doi.org/10.1016/j.matpur.2013.06.003}

\bibitem{BEM}
Sauter, S.A., Schwab, C.: Boundary element methods, \emph{Springer Series in
  Computational Mathematics}, vol.~39.
\newblock Springer-Verlag, Berlin (2011).
\newblock \doi{10.1007/978-3-540-68093-2}.
\newblock \urlprefix\url{https://doi.org/10.1007/978-3-540-68093-2}.
\newblock Translated and expanded from the 2004 German original

\bibitem{SZ90}
Scott, L.R., Zhang, S.: Finite element interpolation of nonsmooth functions
  satisfying boundary conditions.
\newblock Math. Comp. \textbf{54}(190), 483--493 (1990).
\newblock \doi{10.2307/2008497}.
\newblock \urlprefix\url{https://doi.org/10.2307/2008497}

\bibitem{ST10}
Stinga, P.R., Torrea, J.L.: Extension problem and {H}arnack's inequality for
  some fractional operators.
\newblock Comm. Partial Differential Equations \textbf{35}(11), 2092--2122
  (2010).
\newblock \doi{10.1080/03605301003735680}.
\newblock \urlprefix\url{https://doi.org/10.1080/03605301003735680}

\bibitem{Zaslavsky02}
Zaslavsky, G.M.: Chaos, fractional kinetics, and anomalous transport.
\newblock Phys. Rep. \textbf{371}(6), 461--580 (2002).
\newblock \doi{10.1016/S0370-1573(02)00331-9}.
\newblock \urlprefix\url{https://doi.org/10.1016/S0370-1573(02)00331-9}

\end{thebibliography}
\end{document}